\documentclass[12pt]{article}
\usepackage{verbatim,color,amssymb,epsfig}
\usepackage{graphicx,setspace,float,rotating,amsmath,longtable,epsfig,color,lscape,subfigure,url}
\usepackage{epstopdf,multirow,lscape,comment}
\usepackage{latexsym}
\usepackage{amsthm}
\usepackage{amsfonts,bm,amsbsy,bbm}
\usepackage{epsfig}
\usepackage{graphicx,rotating,lscape}
\usepackage{verbatim}
\usepackage{geometry}
\usepackage{setspace}
\usepackage{multirow}
\usepackage[usenames,dvipsnames]{xcolor}
\usepackage{tabularx}
\usepackage[authoryear]{natbib}
\usepackage{array,mathabx}
\usepackage[hidelinks]{hyperref}
\usepackage[normalem]{ulem}
\usepackage[percent]{overpic}
\usepackage{colortbl}
\usepackage{tikz}
\usepackage{enumitem}
\setlist[itemize]{noitemsep, nolistsep}
\usepackage{algorithm,algorithmic}
\usepackage{hyperref}
\usepackage{centernot}

\definecolor{babyblue}{rgb}{0.63, 0.79, 0.95}

\newcolumntype{H}{>{\setbox0=\hbox\bgroup}c<{\egroup}@{}}

\newtheorem{Th}{\underline{\bf Theorem}}

\newtheorem{Pro}{Proposition}
\newtheorem{Lem}{\underline{\bf Lemma}}

\def\S{{\bf S}}

\def\Z{{\bf Z}}
\def\a{{\bf a}}
\def\bfb{{\bf b}}

\def\z{{\bf z}}
\def\0{{\bf 0}}
\def\boxit#1{\vbox{\hrule\hbox{\vrule\kern6pt
          \vbox{\kern6pt#1\kern6pt}\kern6pt\vrule}\hrule}}
\def\wt{\widetilde}
\def\sumi{\sum_{i=1}^n}

\def\wh{\widehat}

\def\log{\hbox{log}}

\def\bse{\begin{eqnarray*}}
\def\ese{\end{eqnarray*}}
\def\be{\begin{eqnarray}}
\def\ee{\end{eqnarray}}
\def\bq{\begin{equation}}
\def\eq{\end{equation}}
\def\bse{\begin{eqnarray*}}
\def\ese{\end{eqnarray*}}
\def\pr{\hbox{pr}}
\def\wh{\widehat}
\def\trans{^{\rm T}}

\def\b1e{{\mathbf e}}

\def\bS{{\mathbf S}}

\def\eff{_{\rm eff}}
\def\n{\nonumber}
\def\bb{{\boldsymbol\beta}}

\def\bA{{\mathbf A}}

\def\b1e{{\mathbf e}}

\def\bS{{\mathbf S}}

\newcommand{\bxi}{\mbox{\boldmath $\xi$}}

\def\bA{{\mathbf A}}

\def\b1e{{\mathbf e}}

\def\bS{{\mathbf S}}

\def\bo{{\mathbf o}}

\def\pr{\hbox{pr}}
\def\wh{\widehat}
\def\trans{^{\rm T}}

\def\b1e{{\mathbf e}}

\def\bq{{\mathbf q}}

\def\bS{{\mathbf S}}

\newcommand{\bphi}{\mbox{\boldmath $\phi$}}

\def\bq{{\boldsymbol{q}}}

\newcommand{\indep}{\rotatebox[origin=c]{90}{$\models$}}

\begin{document}

\begin{center}
{\LARGE{\bf Super doubly robust and efficient estimator for informative covariate censoring}} 
\end{center}

\vskip 2mm
\begin{center}
Zhewei Zhang$^{1}$, Yanyuan Ma$^{1}$, Karen Marder$^2$, and Tanya P. Garcia$^3$
\\
$^1$Department of Statistics, Penn State University, State College, PA 16802\\
$^2$Department of Neurology, Columbia University Medical Center, New York, NY 10032\\
$^3$Department of Biostatistics, Gillings School of Global Public Health, University of North Carolina at Chapel Hill, Chapel Hill, NC, 27516\\ 
\end{center}

\begin{abstract}
\noindent

Early intervention in neurodegenerative diseases requires
  identifying periods before diagnosis when decline is rapid enough to
  detect whether a therapy is slowing progression. Since rapid decline
  typically occurs close to diagnosis, identifying these periods
  requires knowing each patient's time of diagnosis. Yet many patients
  exit studies before diagnosis, making time of diagnosis
  right-censored by time of study exit---creating a right-censored
  covariate problem when estimating decline. Existing estimators
  either assume noninformative covariate censoring, where time of
  study exit is independent of time of diagnosis, or allow informative
  covariate censoring, but require correctly specifying how these
  times are related. We developed SPIRE (Semi-Parametric Informative
  Right-censored covariate Estimator), a super doubly robust estimator
  that remains consistent without  correctly specifying densities
  governing time of diagnosis or time of study exit. Typical double
  robustness requires at least one density to be correct; SPIRE
  requires neither. When both densities are correctly specified, SPIRE
  achieves semiparametric efficiency. We also developed a test for
  detecting informative covariate censoring. Simulations with 85\%
  right-censoring demonstrated SPIRE's robustness, efficiency and reliable detection of informative covariate censoring. Applied to Huntington disease data, SPIRE handled informative covariate censoring appropriately and remained consistent regardless of density specification, providing a reliable tool for early intervention.
\end{abstract}



\clearpage

\section{Introduction}

Slowing or halting neurodegenerative diseases before irreversible damage remains the holy grail of therapeutic development \citep{Tabrizietal2022-hd, Benataretal2022,Jucker2023}. Achieving this goal requires identifying the optimal intervention window---when disease-related decline is rapid enough to determine whether a therapy is slowing or halting decline. This decline is measured by how quickly cognitive test scores worsen, motor abilities deteriorate, and/or brain volume shrinks over time.
However, measuring this decline becomes problematic when data from different patients are compared using arbitrary timepoints, like age or study enrollment date, because patients at the same  age or enrollment date may be at completely different points in their disease---with some declining rapidly and others stable.  When we average across such misaligned data, we obscure the true patterns of decline. Instead, by  anchoring data to diagnosis---when all patients meet the same clinical criteria---researchers can evaluate how clinical measures change as patients approach this point,  revealing when decline is most rapid \citep{DempseyMcCullagh2018,kong2018conditional,chu2020stochastic, scahill2020biological}. 

This anchoring strategy, however, faces a major obstacle:
neurodegenerative diseases progress slowly over decades, so many
patients exit studies before reaching diagnosis. Their time of
diagnosis becomes right-censored---known only to occur after study
exit. The very covariate needed for temporal alignment is
right-censored, creating a statistical problem in which understanding how
clinical measures change as patients approach diagnosis requires using
a right-censored covariate.

Most statistical estimators that handle right-censored covariates
assume \textit{noninformative covariate censoring}: the time of
diagnosis $X$ is independent of the time of study exit $C$ given fully
observed covariates $\Z$ (such as sex and genetic markers); written as
$X\ \indep\ C\mid \Z$ \citep{kong2018conditional,
  zhang2018mechanistic, atem2019improved, lotspeich2024making,
  lee2024robust}.  This assumption is plausible when right-censoring
is administrative, like when a study ends due to loss of funding or
fixed calendar cutoff dates. Yet, this assumption is implausible when right-censoring is due to practical burdens that increase as their disease progresses. For example, patients nearing diagnosis may have mobility limitations or fatigue, making them more like to miss follow-up visits or withdraw from the study altogether.  These reasons provide important, additional information that should not be ignored. The reasons link the time of study exit $C$ to time of diagnosis $X$, even after
adjusting for $\Z$, creating \textit{informative covariate censoring}
($X\centernot\indep C \mid \Z$). Ignoring this dependence by assuming noninformative covariate censoring causes estimators to underestimate how rapidly decline occurs before diagnosis, preventing researchers from correctly identifying intervention windows.

Estimators to address informative covariate censoring have recently
emerged, but they are largely adaptations of  existing estimators
designed for noninformative covariate censoring
\citep{vazquez2024establishing}.  These estimators require modeling
how $X$ and $C$ depend on each other given $\Z$---captured by the
joint density $f_{C,X|\Z}$, or equivalently by $f_{C|\Z}$ and
$f_{X|C,\Z}$.  Yet modeling  $f_{C|\Z}$ and
$f_{X|C,\Z}$ is challenging because researchers observe either the time of diagnosis (when patients are diagnosed
before study exit) or the time of study exit (when they exit the study before a diagnosis is made), but never both together in the same patient. Without observing both times, researchers have no
direct evidence about their joint dependency, making it impossible to
validate any modeling assumptions. Even nonparametric approaches
cannot solve this limitation because estimating the joint density
requires joint observations of both the time of diagnosis and
time of study exit, which informative covariate censoring prevents
\citep{Little2019}. 

A more reliable estimator would handle informative covariate censoring without requiring researchers to model dependencies that are impossible for them to observe. To fill this gap,  we developed the \underline{S}emi-\underline{P}arametric \underline{I}nformative \underline{R}ight-censored covariate \underline{E}stimator (SPIRE). SPIRE  removes the need to model the censoring density $f_{C|\Z}$ and remains consistent even when the only required density, $f_{X|C,\Z}$, is misspecified.  SPIRE thus achieves super double robustness---consistency without needing either $f_{C|\Z}$ or $f_{X|C,\Z}$ to be correctly specified---in contrast to typical double robustness, which needs at least one of these two densities to be correct. When $f_{X|C,\Z}$ is correctly specified, SPIRE is also semiparametric efficient, achieving the lowest possible variance among all consistent estimators under the same modeling assumptions. This combination of super double robustness and semiparametric efficiency allows SPIRE to deliver reliable estimates that capture when decline occurs most rapidly, helping researchers identify optimal intervention windows despite modeling uncertainties.

Moreover, SPIRE works in under both types of covariate censoring (noninformative and informative); researchers do not need to know in advance which type applies to their data. However, distinguishing between the two types
has practical value: when covariate censoring is truly noninformative
($X\ \indep\  C\mid \Z$), simpler estimators achieve greater
efficiency by  avoiding the uncertainty of estimating unnecessary
dependency parameters. Thus, we also developed a test to detect whether covariate censoring is noninformative or informative, allowing researchers to select of the most suitable estimator for their data.

 \section{A class of consistent estimators}
\label{sec:class-consistent-estimators}

\subsection{Model assumptions and identifiability}
\label{sec:notation}

We construct SPIRE under two main assumptions: informative covariate
censoring, $C\centernot\indep X \mid \Z$, which allows time of study exit
$C$ to depend on time of diagnosis $X$ given fully observed covariates
$\Z$, and conditional independence $C \indep Y \mid X, \Z$, which
states that the time of study exit is unrelated to clinical measures $Y$ once
$X$ and $\Z$ are known. Here, $Y$ represents clinical measures such as
cognitive scores, motor abilities, or brain volumes, whose slopes
reveal when decline occurs most rapidly.

At first glance, these two assumptions may appear contradictory: if clinical measures depend on time of diagnosis, and time of study exit depends on time of diagnosis, why would time of study exit not also depend on clinical measures? Yet this combination is plausible in neurodegenerative diseases where the factors driving study exit differ from the clinical measures researchers track.
For instance, in Huntington disease, a genetically inherited disorder, 
a diagnosis is made primarily based on motor signs, while the clinical measures $Y$ used in trials are often composite scores from motor, cognitive, and functional assessments \citep{schobel2017motor}. Time of study exit $C$ depends on time of diagnosis $X$ because patients closer to diagnosis face greater practical burdens (mobility limitations, fatigue, difficulty attending study visits) that make study exit more likely. However, among patients with the same time of diagnosis $X$ and identical covariates $\Z$, variation in their composite scores $Y$  does not predict study exit beyond what $X$ already explains.  Once $X$ and $\Z$ are known, $Y$ provides no additional information about time of study exit, supporting the conditional independence assumption.

Under these assumptions, our goal is to estimate the parameter vector
$\bb$ in the parametric model $f_{Y|X,\Z}(y,x,\z;\bb)$, where $\bb$
characterizes  how rapidly clinical measures decline as patients approach their time of diagnosis
Under right-censoring, we do not observe
time of diagnosis $X$ directly. Instead, we observe $W = \min(X, C)$,
the minimum of time of diagnosis and time of study exit, and the censoring
indicator $\Delta = I(X \leq C)$, which equals 1 when diagnosis occurs
before study exit and 0 otherwise, while covariates $\Z$ remain fully
observed. Thus we must estimate $\bb$ using only the observed data
$(Y, W, \Delta, \Z)$. 

Based on the observed data $(Y, W, \Delta, \Z)$ and our modeling
framework, the likelihood for a single observation takes the form: 
\be\label{eq:model}
&&\left\{\int_{w}^{\infty}f_{Y|X,\Z}(y,w,\z;\bb)f_{C,X|\Z}(c,w,\z)dc\right\}^{\delta}\n\\
&&\left\{\int_{w}^{\infty}f_{Y|X,\Z}(y,x,\z;\bb)f_{C,X|\Z}(w,x,\z)dx\right\}^{1-\delta}f_{\Z}(\z).
\ee
Correctly specifying the joint density $f_{C,X|\Z}$ in
\eqref{eq:model} is problematic because researchers never observe time
of diagnosis and time of study exit together in the same patient. Rather than risk
misspecifying this density---which would bias estimates of how rapidly clinical measures decline as patients approach diagnosis---we leave
$f_{C,X|\Z}$ unspecified. We also leave  $f_{\Z}$ unspecified to
maintain full flexibility. Despite not specifying  functional forms
for these densities, Lemma \ref{lem:1} establishes that both
densities, along with $\bb$,  remain identifiable from the observed
data (proof in  Section \ref{proofoflem1}, Supplementary Material).  
\begin{Lem}\label{lem:1}
  Suppose $f_{Y|X,\Z}(y,x,\z,\bb)$ satisfies the completeness
  condition: if a function $g(x,\z)$ satisfies
$\int f_{Y|X,\Z}(y,x,\z,\bb)g(x,\z)dx=0$ for all $y,\z$, then
$g(x,\z)=0$. Under this condition, all components in
the likelihood \eqref{eq:model}, i.e., $f_{C,X|\Z}$, $f_{\Z}$, $\bb$, are
identifiable.
\end{Lem}
In practice, the completeness condition must be verified case by case. For example, when $Y\mid X,\Z$ follows a normal distribution with mean $\beta_0+\beta_1X+ a(\Z)$ for any function $a(\cdot)$ and standard deviation $\sigma$, the completeness condition can be verified using techniques from Laplace transform theory \citep{chareka2007finite}. These techniques show that the only function $g(x,\z)$ satisfying $\int f_{Y|X,\Z}(y,x,\z,\bb)g(x,\z)dx=0$ for all $y,\z$ is $g(x,\z)=0$. Thus identifiability is guaranteed for this commonly-used model.

\subsection{Constructing the class of consistent estimators}
\label{sec:derive-consistent-class}

To construct consistent (asymptotically unbiased) estimators for $\bb$
while leaving $f_{C,X|\Z}$ and $f_{\Z}$ unspecified, we use  a
  geometric approach \citep{bickel1993efficient,
  tsiatis2006semiparametric}. This approach treats the unspecified
densities as nuisance parameters and provides a framework to
separate their influence from the estimation of $\bb$. At the heart of
this approach are two key subspaces derived from the likelihood in
\eqref{eq:model}. The first is the nuisance tangent space $\Lambda$,
which contains all functions associated with the nuisance parameters
$f_{C,X|\Z}$ and $f_{\Z}$; its explicit form and derivation appear in
Section \ref{proofofpro1}. The second is its orthogonal complement
$\Lambda^{\perp}$, which by construction contains functions that are
orthogonal to---and thus minimally influenced by---the nuisance
parameters. Functions in $\Lambda^{\perp}$ have mean zero (proof in
Section \ref{proofofpro2}), which  allows them to form unbiased
estimating equations. Under mild regularity conditions
\citep{foutz1977unique}, solving these equations yields consistent
estimators of $\bb$.  
\begin{Pro}\label{pro:LambdaP}
    The orthogonal complement $\Lambda^\perp$ takes the form
\bse
\Lambda^\perp&=&\left[\bfb(y,w,\delta,\z)
=\delta\bfb_1(y,x,\z)+(1-\delta)\bfb_0(y,c,\z):
 E\{\bfb(Y,w,\Delta,\z) \mid c, x,\z\}=\0\right].
\ese
\end{Pro}
Any function in  $\Lambda^\perp$ yields a consistent estimator. By
varying  the choices of  $\bfb_1$ and $\bfb_0$  in Proposition
\ref{pro:LambdaP}, we now derive familiar estimators and examine their
properties. To facilitate this derivation, we introduce notation for
two score vectors: $\bS_{\bb}(y,w,\delta,\z;\bb) \equiv \partial \log
f_{Y,W,\Delta,\Z}(y,w,\delta,\z;\bb)/\partial\bb$ denotes the score
vector for $\bb$ from the observed likelihood in \eqref{eq:model}, and
$\bS^F_\bb(y,x,\z;\bb) \equiv \partial \log
f_{Y|X,\Z}(y,x,\z;\bb)/\partial\bb$ denotes the score vector from the
conditional density $f_{Y|X,\Z}$.

\begin{enumerate}
\item {\bf Complete case (CC) estimator:}
Setting $\bfb_0=\0$  and $\bfb_1(y,x,\z)=\S^F_{\bb}(y,x,\z)$  (which
satisfies $E\{\bfb_1(Y,x,\z)\mid x,\z\}=\0$) yields the CC estimator,
obtained by solving $\sum_{i=1}^n
\delta_i\S_{\bb}^F(y_i,x_i,\z_i)=\0$. While simple to implement and
consistent even when $f_{C,X|\Z}$ and $f_{\Z}$ are misspecified, this
estimator suffers from substantial efficiency loss by discarding all
censored observations. 

\item {\bf Inverse probability weighting  (IPW) estimator:}
Setting $\bfb_0=\0$ and $\bfb_1(y,x,\z)=\S^F_{\bb}(y,x,\z)/\pr(C\geq
x|x,\z)$ yields the IPW estimator, obtained by solving $\sum_{i=1}^n
\delta_i\S_{\bb}^F(y_i,x_i,\z_i)/\pr(C\geq x_i|x_i,\z_i)=\0$.  
 The IPW estimator improves upon the CC estimator by weighting
 uncensored observations to approximate what the full sample would
 look like without right-censoring. Like the CC estimator,  the IPW
 estimator remains consistent even when $f_{C,X|\Z}$ and $f_{\Z}$ are
 misspecified but suffers the same efficiency loss from discarding
 censored observations.

\item {\bf  Maximum likelihood estimator (MLE):}
The MLE does not arise from choices within $\Lambda^\perp$ but instead
maximizes the likelihood in \eqref{eq:model}, obtained by solving
$\sum_{i=1}^n \left[\delta_i\S_{\bb}^F(y_i,x_i,\z_i) +
  (1-\delta_i)E\{\S_{\bb}^F(y_i,X,\z_i)\mid X>c_i,
  y_i,c_i,\z_i\}\right]=\0$. The MLE incorporates all data, achieving
maximum efficiency when  $f_{C,X|\Z}(c,x,\z)$ is
correctly specified, but this density cannot be validated since $C$
and 
$X$ are never observed together, making the MLE prone to bias.
\end{enumerate}
Each estimator forces an unnecessary trade-off: sacrifice efficiency
by discarding censored observations (CC, IPW) or risk bias by
requiring correct specification of 
$f_{C,X|\Z}$ (MLE). These trade-offs motivated our development of SPIRE.

\section{The super doubly robust and efficient estimator}
\label{sec:spire}

\subsection{Development and properties of SPIRE}

Having established that $\Lambda^\perp$  can provide an entire
class of consistent estimators, a natural question arises: is there an
optimal choice within this class? The answer is yes—the semiparametric
efficient estimator, which achieves the smallest possible variance
among all consistent estimators in our framework. Finding this
estimator hinges on the projection theorem \citep{bickel1993efficient, tsiatis2006semiparametric}. We take the score vector $\bS_{\bb}(y,w,\delta,\z;\bb)$
defined earlier and project it onto $\Lambda^\perp$. Geometrically,
this projection finds the element in $\Lambda^\perp$ closest to
$\bS_{\bb}$. The resulting element, called the efficient score vector, 
$\S_{\eff}$, retains the maximum information about $\bb$ while lying
in the orthogonal complement, yielding an estimator with the smallest
variance achievable. The following proposition specifies the form of
this efficient score vector (proof in Section \ref{proofofpro3}).
\begin{Pro}\label{pro:Seff}
    The efficient score vector for $\bb$ is
    \bse
\S\eff(y,w,\delta,\z)&=&
\delta\S_\bb^F(y,x,\z;\bb)
-\delta 
\frac{E\{\a_0(C,x,\z)I(x\le C)\mid x,\z\}}{E\{I(x\le C)\mid x,\z\}}\\
&&-(1-\delta) \frac{E[\{\a_0(c,X,\z)-\S_\bb^F(y,X,\z;\bb)\}I(X> c)\mid y, c,\z]}{E\{I(X>
  c)\mid y, c,\z\}},
\ese
where $\S_\bb^F(y,x,\z;\bb)= \partial \log
f_{Y|X,\Z}(y,x,\z;\bb)/\partial \bb$ and $\a_0(c,x,\z)$ satisfies  
\bse
&&I(x\le c)\frac{E\{\a_0(C,x,\z)I(x\le C)\mid x,\z\}}{E\{I(x\le C)\mid x,\z\}}\n\\
&&+I(c<x)E\left[\frac{E\{(\a_0(c,X,\z)-\S_\bb^F(Y,X,\z))I(X> c)\mid Y, c,\z\}}{E\{I(X>
	c)\mid Y, c,\z\}}\mid c, x,\z\right]=0.
\ese
\end{Pro}
We use this efficient score vector to construct SPIRE, denoted as 
$\wh\bb_n$, which solves $\sum_{i=1}^n
\S_{\eff}(y_i,w_i,\delta_i,\z_i;\bb_n)=\0$.  Since $\S_{\eff}$
involves expectations with respect to $f_{C,X|\Z}$, we now examine the
consequences of using a working model $f_{C,X|\Z}^*$ in place of the
true density. 

Let $\S_{\eff}^*$ denote the efficient score vector obtained under this working
model. It has the same form as Proposition \ref{pro:Seff}, but with
expectations $E$ replaced by $E^*$ (computed under $f_{C,X|\Z}^*$) and
$\a_0(c,x,\z)$ replaced by $\a_0^*(c,x,\z)$. To establish the
asymptotic properties of SPIRE under this working model, we introduce 
$J_n^*(\bb)=n^{-1}\sum_{i=1}^{n}{\partial\S\eff^*(y_i,w_i,\delta_i,\z_i;\bb)}/{\partial
  \bb\trans}$,
$J^*(\bb)=E\left[{\partial\S\eff^*(Y,W,\Delta,\Z;\bb)}/{\partial
    \bb\trans}\right]$ and
$V^*(\bb)=E\{\S\eff^*(Y,W,\Delta,\Z;\bb)^{\otimes2}\}$. We also impose
the following standard regularity conditions that  
 ensure $\S_{\eff}^*$ has a unique solution and well-behaved derivatives \citep{newey1994}: 
 \renewcommand{\labelenumi}{(C\arabic{enumi})} 
 \renewcommand{\theenumi}{C\arabic{enumi}}     
 
 \begin{enumerate}
 	\item\label{assump:c2} $\bb_0 \in \mathcal{B}$, and $\mathcal{B}$ is compact.
 	\item\label{assump:c1} On $\mathcal B$, $E\{\S\eff^*(Y,W,\Delta,\Z;\bb)\}=0$ only if
 	$\bb=\bb_0$, where $\bb_0$ is the true value of the parameter.
 	\item\label{assump:c3} $\S\eff^*(y,w,\delta,\z;\bb)$ is
          continuous in $\bb$ on $\mathcal{B}$. 
 	\item\label{assump:c4} $E[\sup_{\bb \in \mathcal{B}}\Vert
          \S\eff^*(Y,W,\Delta,\Z;\bb) \Vert] < \infty$. 
 	\item\label{assump:c5} $\bb_0$ lies in the interior of $\mathcal{B}$.
 	\item\label{assump:c6} $\S\eff^*(y,w,\delta,\z;\bb)$ is continuously differentiable
 	in a neighborhood $\mathcal{N}$ of $\bb_0$.
 	\item\label{assump:c7} $E[\sup_{\bb \in \mathcal{B}}\Vert
 	\partial\S\eff^*(Y,W,\Delta,\Z;\bb)/ \partial \bb\trans \Vert] < \infty$.
 	\item\label{assump:c8} $J^*(\bb_0)$ is nonsingular.
 \end{enumerate}

With these regularity conditions in place, we establish SPIRE's
consistency despite misspecification of $f_{C,X|\Z}$ (proof in
Section \ref{sec:proofthmconsis}). 
\begin{Th}[Consistency]{\label{thm:consis}}
	Under regularity conditions (C1)--(C4), if $\wh\bb_n$ solves
        the estimating equation
        $\sum_{i=1}^{n}\S\eff^*(y_i,w_i,\delta_i,\z_i;\wh\bb_n)=\0$
        using  any working model $f_{C,X|\Z}^*(c,x,\z)$, then
        $\wh{\bb}_n \to\bb_0$ in probability. 
\end{Th}
Beyond consistency, we also establish SPIRE's asymptotic distribution and efficiency properties  (proof in
      Section \ref{sec:proofthmnormal}).
\begin{Th}[Asymptotic Normality and Semiparametric Efficiency]{\label{thm:norm}}
	Under regularity conditions (C1) -- (C8), SPIRE satisfies
        $\sqrt{n}(\wh{\bb}_n-\bb_0) \to
        N[\0,J^*(\bb_0)^{-1}V^*(\bb_0)\{J^*(\bb_0)^{-1}\}\trans]$ in
      distribution as $n \to \infty$.    When
      $f_{C,X|\Z}^*(c,x,\z)=f_{C,X|\Z}(c,x,\z)$ (i.e., the working
      model is correctly specified), SPIRE achieves the semiparametric
      efficiency bound with asymptotic variance
      $[E\{\S_{\eff}(Y,W,\Delta,\Z;\bb)^{\otimes2}\}]^{-1}$. 
\end{Th}
Together, these theorems establish SPIRE's defining
  properties. Theorem \ref{thm:consis} shows that SPIRE maintains
  consistency using any working model $f_{C,X|\Z}^*$, even when
  misspecified. Yet, the reason behind this robustness reveals an even
  stronger property. Through the factorization
  $f_{C,X|\Z}^*(c,x,\z)=f^*_{X|C,\Z}(x,c,\z) f^*_{C|\Z}(c,\z)$, we
  show in the next section that $f^*_{C|\Z}(c,\z)$ cancels entirely
  when solving for $\wh\bb_n$. This cancellation means SPIRE achieves
  super double robustness---needing neither $f_{C|\Z}^*$ nor
  $f_{X|C,\Z}^*$ to be correctly specified for consistency---while still
  attaining the efficiency bound when $f_{X|C,\Z}^*$  alone is correct.

This dual achievement resolves a longstanding trade-off in the
censored covariate literature. Robust estimators like the CC and the
IPW estimators maintain consistency under misspecification but
sacrifice efficiency by discarding censored observations, while
efficient estimators like the MLE require correct specification of
densities that cannot be validated since researchers never observe
both time of diagnosis and time of study exit together in the same
patient. SPIRE offers both: when $f_{X|C,\Z}$ is correctly
specified, SPIRE achieves the semiparametric efficiency bound; when
$f_{X|C,\Z}$ is misspecified, SPIRE sacrifices some efficiency but
remains consistent---unlike the MLE which becomes biased. This
consistency guarantee means different research groups can analyze the
same neurodegenerative disease cohort with different working models
for $f_{X|C,\Z}$ and still obtain valid estimates of pre-diagnosis
decline patterns. When their working models are correct, they also
gain optimal statistical power to identify when decline is most rapid,
combining reproducibility with the ability to detect intervention
windows despite  of the inherent uncertainty of right-censored covariate
settings.

\subsection{Implementation of SPIRE}

Computing $\S_{\text{eff}}$ in Proposition \ref{pro:Seff}  requires
evaluating the implicitly-defined function $\a_0^*(c,x,\mathbf{z})$
within nested conditional expectations. We now derive tractable
expressions for $\a_0^*(c,x,\mathbf{z})$. 

Differentiating the likelihood in \eqref{eq:model} with respect to
$\bb$ gives the score vector under the working model: 
 $\S_{\bb}(y,w,\delta,\z,\bb)=\delta\S_\bb^F(y,x,\z;\bb)+(1-\delta)\frac{E^*\{I(c<X)\S_\bb^F(y,X,\z;\bb) 
   \mid y,c,\z\}}{E^*\{I(c<X) \mid y,c,\z\}}$.  The main  insight is
 that $\S_{\text{eff}}$ equals $\S_{\bb}^*$ minus correction terms
 involving $\a_0^*(c,x,\mathbf{z})$ (see Proposition
 \ref{pro:Seff}). Therefore, $\a_0^*(c,x,\mathbf{z})$ must be chosen
 so that $E\{\S_{\bb}^*(Y,w,\delta,\z;\bb)\mid c,x,\z\} =
 E\{\text{correction terms with } \a_0^*(c,x,\mathbf{z})\mid
 c,x,\z\}.$ 
Solving for $\a_0^*(c,x,\mathbf{z})$ that satisfies this equality yields:
\be
E\{\S_\bb^*(Y,w,\delta,\z,\bb)\mid c, x,\z\}
&=&I(x\le c)\frac{E^*\{\a_0^*(C,x,\z)I(x\le C)\mid x,\z\}}{E^*\{I(x\le
  C)\mid x,\z\}}\label{eqn:explicit-a0}\\ 
&+&I(c<x)E\left[\frac{E^*\{\a_0^*(c,X,\z)I(X> c)\mid Y,
    c,\z\}}{E^*\{I(X>  c)\mid Y, c,\z\}}\mid c, x,\z\right].\nonumber 
\ee
Expressing \eqref{eqn:explicit-a0} in integral form allows us to derive $\a_0^*(c,x,\mathbf{z})$:
\be\label{eq:int}
&&I(c<x) 
\int
\frac{
	\int_c^{\infty}\S^F_{\bb}(y,x,\z;\bb) f_{Y|X,\Z}(y,x,\z;\bb)f_{C,X|\Z}^*(c,x,\z)dx}
{\int_{c}^{\infty}f_{Y|X,\Z}(y,x,\z;\bb)f_{C,X|\Z}^*(c,x,\z)dx}
f_{Y|X,\Z}(y,x,\z;\bb)dy\n\\
&=&\frac{\int_x^\infty \a_0^*(c,x,\z)f_{C,X|\Z}^*(c,x,\z)dc}{\int_x^\infty f_{C,X|\Z}^*(c,x,\z)dc}
I(x\le c) +\\
&&I(c<x)\int\frac{
	\int_{c}^{\infty}\a_0^*(c,x,\z)f_{Y|X,\Z}(y,x,\z;\bb)f_{C,X|\Z}^*(c,x,\z)dx
}{\int_{c}^{\infty}f_{Y|X,\Z}(y,x,\z;\bb)f_{C,X|\Z}^*(c,x,\z)dx}
f_{Y|X,\Z}(y,x,\z;\bb)dy. \n
\ee

Two simplifications transform  \eqref{eq:int} into a tractable
expression for $\alpha_0^*$. First, when $x\leq c$, \eqref{eq:int}
simplifies to  
\bse
0=\frac{\int_x^\infty\a_0^*(c,x,\z)f_{C,X|\Z}^*(c,x,\z)dc}{\int_x^\infty f_{C,X|\Z}^*(c,x,\z)dc},
\ese
yielding $\a_0^*(c,x,\z)=\0$. Thus, we need only determine
$\a_0^*(c,x,\z)$ for $x>c$.  Second, for $x>c$, the factorization
$f_{C,X|\Z}^*(c,x,\z)=f^*_{X|C,\Z}(x,c,\z) f^*_{C|\Z}(c,\z)$ allows us
to cancel $f^*_{C|\Z}(c,\z)$ throughout  \eqref{eq:int}.  
After applying these two simplifications, we obtain:
\bse
\S\eff^*(y,w,\delta,\z)&=&
\S_\bb^*(y,w,\delta,\z)
-\delta
\frac{\int_x^\infty\a_0^*(c,x,\z)f_{X|C,\Z}^*(x,c,\z)dc}{\int_x^\infty f_{X|C,\Z}^*(x,c,\z)dc}\\
&&-(1-\delta)\frac{
\int_{c}^{\infty}\a_0^*(c,x,\z)f_{Y|X,\Z}(y,x,\z;\bb)f_{X|C,\Z}^*(x,c,\z)dx
}{\int_{c}^{\infty}f_{Y|X,\Z}(y,x,\z;\bb)f_{X|C,\Z}^*(x,c,\z)dx},
\ese
where $\a_0^*(c,x,\z)$ satisfies 
\bse
&&I(c<x) 
\int
\frac{
	\int_{c}^{\infty}\S^F_{\bb}(y,x,\z;\bb)f_{Y|X,\Z}(y,x,\z;\bb) f_{X|C,\Z}^*(x,c,\z) dx}
{\int_{c}^{\infty}f_{Y|X,\Z}(y,x,\z;\bb)f_{X|C,\Z}^*(x,c,\z)dx}
f_{Y|X,\Z}(y,x,\z;\bb)dy\n\\
&=&I(c<x)\int\frac{
	\int_{c}^{\infty}\a_0^*(c,x,\z)f_{Y|X,\Z}(y,x,\z;\bb)f_{X|C,\Z}^*(x,c,\z)dx
}{\int_{c}^{\infty}f_{Y|X,\Z}(y,x,\z;\bb)f_{X|C,\Z}^*(x,c,\z)dx}
f_{Y|X,\Z}(y,x,\z;\bb) dy.
\ese

The simplified expressions show that we only need to model
 $f_{X|C,\Z}^*(x,c,\z)$ directly, as $f^*_{C|\Z}(c,\z)$ has canceled entirely from the implementation. Furthermore, we can
approximate  $f_{X|C,\Z}^*(x,c,\z)$ with a discrete density:
$f_{X|C,\Z}^* \approx \sum_{j=1}^m
p_j(c,\z)I(x=x_j)$, where we place mass at $m$ grid points $0 \leq x_1
< \cdots < x_m \leq \max(w_i)$ with weights
$p_j(c,\z)=f_{X|C,\Z}^*(x_j,c,\z)/\sum_{k=1}^m
f_{X|C,\Z}^*(x_k,c,\z)$. With this discretization, we have
\bse
&&E\left[\frac{E^*\{\a_0^*(c,X,\z)I(X> c)\mid Y, c,\z\}}{E^*\{I(X>
	c)\mid Y, c,\z\}}\mid c, x_k,\z\right] \\
&\approx& \int \left\{ \frac{\sum_{j=1}^{m} \a_0^*(c,x_j,\z)I(c<x_j)p_j(c,\z)f_{Y|X,\Z}(y,x_j,\z;\bb)}{\sum_{j=1}^{m}I(c<x_j)p_j(c,\z)f_{Y|X,\Z}(y,x_j,\z;\bb)}\right\} f_{Y|X,\Z}(y,x_k,\z;\bb) dy,
\ese
and 
\bse
&&E\left[\frac{E^*\{S^F_{\bb}(Y,X,\z;\bb)I(X> c)\mid Y, c,\z\}}{E^*\{I(X>
	c)\mid Y, c,\z\}}\mid c, x_k,\z\right] \\
&\approx&\int \left\{ \frac{\sum_{j=1}^{m} S^F_{\bb}(y,x_j,\z;\bb)I(c<x_j)p_j(c,\z)f_{Y|X,\Z}(y,x_j,\z;\bb)}{\sum_{j=1}^{m}I(c<x_j)p_j(c,\z)f_{Y|X,\Z}(y,x_j,\z;\bb)}\right\} f_{Y|X,\Z}(y,x_k,\z;\bb) dy.
\ese
This discretization transforms our problem of finding $\a_0^*(c,x,\z)$
into a system of linear equations: 
\begin{equation}\label{eq:imple}
\mathbf{A}(c,\mathbf{z})\a^{\trans}(c,\mathbf{z}) = \mathbf{b}^{\trans}(c,\mathbf{z}).
\end{equation}
Here, $\a(c,\mathbf{z})$ is the $q \times m$ matrix containing the
unknown values $\{\a_0^*(c,x_1,\mathbf{z}), \cdots,
\a_0^*(c,x_m,\mathbf{z})\}$ that we seek. The matrix
$\mathbf{A}(c,\mathbf{z})$ is $m \times m$ with $(k,j)$-th element: 
\begin{align*}
A_{kj}(c,\mathbf{z}) = \int \left\{ \frac{
  I(c<x_j)p_j(c,\mathbf{z})f_{Y|X,\mathbf{Z}}(y,x_j,\mathbf{z};\bb)}{\sum_{\ell=1}^{m}I(c<x_\ell)p_\ell(c,\mathbf{z})f_{Y|X,\mathbf{Z}}(y,x_\ell,\mathbf{z};\bb)}\right\}
  f_{Y|X,\mathbf{Z}}(y,x_k,\mathbf{z};\bb) dy,
\end{align*}
and the matrix $\mathbf{b}(c,\mathbf{z})$ is $q \times m$ with $k$-th column:
\begin{align*}
\int \left\{ \frac{\sum_{j=1}^{m}
    \S^F_{\bb}(y,x_j,\mathbf{z};\bb)I(c<x_j)p_j(c,\mathbf{z})f_{Y|X,\mathbf{Z}}(y,x_j,\mathbf{z};\bb)}{\sum_{\ell=1}^{m}I(c<x_\ell)p_\ell(c,\mathbf{z})f_{Y|X,\mathbf{Z}}(y,x_\ell,\mathbf{z};\bb)}\right\}
f_{Y|X,\mathbf{Z}}(y,x_k,\mathbf{z};\bb) dy.
\end{align*}
Algorithm \ref{alg:spire} summarizes the complete SPIRE implementation.

\begin{algorithm}[H]
	\caption{SPIRE Implementation}
	\label{alg:spire}
    1: Approximate
	$f_{X|C,\Z}^*(x,c,\z)$ as $\sum_{j=1}^{m}p_j(c,\z)I(x=x_j)$,  where
	$x_j, j=1, \dots, m$ are grid points evenly spread on $[0, \max(w_i)]$.

	2: For each $i=1, \dots, n$:
	
	 if $\delta_i=1$, let $\S\eff^*(y_i,w_i,\delta_i,\z_i)=\S_\bb^*(y_i,w_i,\delta_i,\z_i)$; 
	 
	 if $\delta_i=0$, let
	 \bse
	 \S\eff^*(y_i,w_i,\delta_i,\z_i)=\S_\bb^*(y_i,w_i,\delta_i,\z_i)
	 -\frac{\sum_{j=1}^{m} \a_0^*(c_i,x_j,\z_i)I(c_i<x_j)p_j(c_i,\z_i)f_{Y|X,\Z}(y_i,x_j,\z_i;\bb)}{\sum_{j=1}^{m}I(c_i<x_j)p_j(c_i,\z_i)f_{Y|X,\Z}(y_i,x_j,\z_i;\bb)},
	 \ese 
	 where $\a_0^*(c_i,x_j,\z_i)$ is obtained from \eqref{eq:imple}.
	 
	 3: Solve the estimation equation  $\sum_{i=1}^{n}\S\eff^*(y_i,w_i,\delta_i,\z_i;\bb)=\0$ 	 to obtain $\wh\bb_n$.
\end{algorithm}

\subsection{Test for noninformative covariate censoring}
\label{sec:noninfo-test}

While SPIRE handles both informative and noninformative covariate
censoring, detecting noninformative covariate censoring ($X \indep C
\mid \Z$) allows the use of more efficient estimators. Thus, we developed a test for this type of detection.

The test exploits how estimators respond differently to misspecifying $f_{X|C,\Z}^*$: SPIRE, the CC estimator, and the IPW estimator remain consistent under misspecification, while the MLE becomes inconsistent.
\begin{Th}[Chi-square Test for Noninformative Covariate Censoring]{\label{thm:chisq}}
Under regularity conditions (C1)--(C8), let $\wh{\bb}_1$   be either
        SPIRE, the CC estimator, or the IPW estimator, and let
        $\wh{\bb}_2$ be the  MLE. 
        	When $f_{X|C,\Z}^*$ is correctly specified, $n(\wh{\bb}_1-\wh{\bb}_2)^{\trans}V^{-1}(\wh{\bb}_1-\wh{\bb}_2) \to \chi^2_p$
	in distribution when $n \to \infty$, where
        $V=var(\bphi_1-\bphi_2)$, $\chi^2_p$ is a chi-square distribution with $p$ degrees of freedom.
       When $f_{X|C,\Z}^*$ is misspecified,
        the asymptotic distribution of
        $n(\wh{\bb}_1-\wh{\bb}_2)^{\trans}V^{-1}(\wh{\bb}_1-\wh{\bb}_2)
        $ is a non-central chi-square distribution.
        Here, $\bphi_1$ and $\bphi_2$ are the influence functions of
        $\wh{\bb}_1$ and $\wh{\bb}_2$, respectively. Specifically,
        $\bphi_i=-[E\{\partial\S_i(Y,W,\Delta,\Z;\bb)/\partial
          \bb\trans\}]^{-1}\S_i(Y,W,\Delta,\Z;\bb)$, for $i=1,2$,
          where $\S_1$ is $\S_{\rm{CC}}$, $\S_{\rm{IPW}}^*$, or $\S\eff$, and  $\S_2$ is
          $\S_{\rm{MLE}}$.  In practice, $V$ is estimated by $\wh V=n^{-1}\sumi\{\bphi_1(Y_i,W_i,\Delta_i,\Z_i;\wh\bb_1)-\bphi_2(Y_i,W_i,\Delta_i,\Z_i;\wh\bb_2)\}^{\otimes2}$.        
\end{Th}
The proof of Theorem \ref{thm:chisq} is  in Section\ref{proofofthmchisq}.
Based on Theorem \ref{thm:chisq}, we construct the test statistic
$T_{\rm chi}\equiv n(\wh{\bb}_1-\wh{\bb}_2)^{\trans}\wh 
V^{-1}(\wh{\bb}_1-\wh{\bb}_2)$. For the working model 
$f_{X|C,\mathbf{Z}}^*$, we use a nonparametric estimator of $f_{X|\Z}$,
such as the localized Kaplan-Meier estimator. Under the null
hypothesis of noninformative covariate censoring ($X \indep C \mid
\mathbf{Z}$), we have $f_{X|C,\Z}^* = f_{X|\Z}$, so the working model is
correctly specified and $T_{\text{chi}}$ follows a $\chi^2_p$
distribution asymptotically. We reject the null hypothesis at
significance level $\alpha$ if $T_{\text{chi}} > \chi_{p,\alpha}^2$,
where $\chi_{p,\alpha}^2$ is the $(1-\alpha)$ quantile of the
chi-square distribution with $p$ degrees of freedom. 

This test allows researchers to determine if the covariate censoring in their data is noninformative. If the test fails to reject the null hypothesis, researchers should consider using simpler estimators that  assume noninformative censoring for improved efficiency. 
 If the test rejects the null hypothesis, researchers should use SPIRE for valid inference despite the informative covariate censoring.

\section{Simulation studies}
\label{sec:sim}

\subsection{Evaluation of robustness and efficiency}

We evaluated SPIRE's super double robustness and efficiency in two
settings: a controlled setting, where $f_{X|C,Z}$ follows a normal
density, and a realistic setting, where $f_{X|C,\Z}$ follows a beta
density calibrated to match the Huntington disease data analyzed in
Section \ref{sec:app}. 

In the controlled setting, we generated $N=1,000$ samples of $n=1,000$
observations, each with $Z \sim \text{Bernoulli}(0.5)$, $C|Z \sim
\text{Uniform}(Z-0.5, Z+0.5)$, and $X|C,Z \sim \text{Normal}\{C-\mu,
(Z+1)/4\}$. The response $Y$ followed the linear model $Y = \beta_0 +
\beta_1X + \beta_2Z + \epsilon$, where $\epsilon\sim
\text{Normal}(0,1)$ and $\bb = (0.5, 0.2, -0.2)\trans$. We varied $\mu
\in \{0.75, 0, -0.3, -0.5\}$ to achieve right-censoring rates of
approximately 10\%, 50\%, 70\%, and 80\%.

In the realistic setting, we calibrated our simulation to the
Huntington disease dataset ($n=3,657$) by generating $N=1,000$ samples
of $n=3,000$ observations. We generated covariates matching the real
data structure: age at study entry $Z_0 \sim \text{Beta}(1.8874,
3.8470)$, cytosine-adenine-guanine (CAG) repeat length  (the genetic mutation causing Huntington disease)  $Z_1 \sim
\text{Beta}(3.5383, 11.4963)$, and sex $Z_2 \sim \text{Bernoulli}(0.5)$. The time of study exit
$C|\Z \sim \text{Beta}(0.3+Z_1, 1.1+Z_2) + Z_0$ and time of diagnosis
$X|C,\Z \sim \text{Beta}(1.6+5C, 2+Z_1+Z_2)+ Z_0$ yielded
approximately 85\% right-censoring to match the observed 84.7\%. The
response $Y|X,\Z$ followed $Y = \beta_0 + \beta_1(X-Z_0) + \beta_2Z_1
+ \beta_3Z_2 + \beta_4(X-Z_0)Z_2 + \epsilon$, where $\epsilon \sim
\text{Normal}(0,\sigma^2)$ and $\bb = (1.3, -1.8, -1.5, 0.1, 0.2,
1)\trans$. The term $(X-Z_0)$ measures years from study entry to
diagnosis, anchoring patients at diagnosis to reveal how clinical
measures accelerate as patients approach diagnosis. 

We implemented four estimators: the CC estimator (which analyzes only
uncensored observations), and three estimators that require a working
model $f_{X|C,\Z}^*$---the IPW estimator, MLE, and SPIRE. For the
latter three, we tested both correctly specified and deliberately
misspecified working models to evaluate robustness under varying
degrees of model violation. In the controlled setting, we tested two
working models for $f_{X|C,Z}^*$: (1) correctly specified as
$f_{X|C,Z}$, and (2) misspecified as uniform over $[\bar{X} - 3s(X),
\bar{X} + 3s(X)]$, where $\bar{X}$ and $s(X)$ denote the sample mean
and standard deviation of $X$. This misspecification ignores the true
dependence of $X$ on both $C$ and $Z$. In the realistic setting, we
implemented three working models for $f_{X|C,\Z}^*$: (1) correctly
specified as $f_{X|C,\Z}$, (2) misspecified as uniform over $[0,1]$,
ignoring all covariate dependencies, and (3) misspecified using a
localized Kaplan-Meier estimator that assumes $X \indep C|\Z$. The
localized Kaplan-Meier estimator uses the derivative of 
\be\label{eq:KM}
\wh S_{X|\Z}(t,\z)=\max \left[\prod_{j=1}^{n} \left\{ 1-
    \frac{K_h(\z-\z_j)}{\sum_{k=1}^{n}I(w_k \geq w_j)K_h(\z-\z_k)}
  \right\}^{I(w_j \leq t,\delta_j=1)},n^{-1}\right],
\ee
with Gaussian kernel $K_h(t)=K(t/h)/h$ and bandwidth $h=0.05$. This
third working model represents a sophisticated yet incorrect
specification---it captures the marginal distribution of $X|\Z$ while
wrongly assuming independence from $C$. 

Tables \ref{table:b1s1} and \ref{table:b1s2} show SPIRE's super double
robustness and semiparametric efficiency. Under correct specification
of $f_{X|C,\Z}^*$, all estimators achieved consistency, with empirical
bias near zero and 95\% confidence interval coverage at nominal
levels.  However, performance diverged  under misspecification: SPIRE
maintained consistency even when the working model was
wrong---achieving super double robustness---while the MLE produced
biased estimates.

The standard errors reveal SPIRE's efficiency advantages. When the
working model $f_{X|C,\Z}^*$ is correctly specified, SPIRE achieves
the semiparametric efficiency bound, producing standard errors
20--41\% smaller than the IPW estimator in the controlled setting and
16--23\% smaller in the realistic setting.  The gains over the CC
estimator were more modest but still meaningful, reaching 12\% at the
highest censoring rates. Interestingly, even under the misspecified
localized Kaplan-Meier estimator that wrongly assumes $X \indep C|\Z$,
SPIRE still outperformed the IPW estimator---a benefit not guaranteed
by theory. 

The MLE's behavior illustrates why robustness matters as much as
efficiency. While the MLE produced the smallest standard errors among
all estimators, this apparent advantage became a liability under
misspecification. The MLE's point estimates were biased, yet its
confidence intervals remained narrow: at 80\% right-censoring with
misspecification, these precise-looking 95\% intervals included the
true parameter values only 52\% of the time. Researchers would thus
report seemingly precise results that are wrong nearly half the
time. In contrast, SPIRE trades narrower intervals for reliability:
its confidence intervals maintain their 95\% coverage even under
misspecification. 

Across both settings, empirical standard deviations closely matched
the average standard errors predicted by our sandwich variance formula
(Theorem \ref{thm:norm}), indicating that SPIRE's uncertainty
quantification remains accurate whether the working model is correctly
specified or not.  The robustness and efficiency patterns shown in
Tables \ref{table:b1s1} and \ref{table:b1s2} hold across all model
parameters (see Tables \ref{table:b0s1}--\ref{table:b2s2} in Section
\ref{sec:addtional-sims}). 

\begin{table}[t]
	\renewcommand{\arraystretch}{1.5}
	\begin{tabular}{ccllllllll}
		\hline
		&               & \multicolumn{4}{c}{10\% censoring}         & \multicolumn{4}{c}{50\% censoring} \\ \cline{3-10} 
		$f^*_{X\mid C,Z}$& Estimator     & \multicolumn{1}{c}{Mean} & ESE & ASE & Cov & Mean    & ESE    & ASE    & Cov    \\ \hline
		true          & SPIRE      &          0.1992                &    0.0577 &  0.0563   &   94.2\%  &  0.2020       &   0.0939     &    0.0938    &    95.5\%    \\
		& CC &           0.2010               & 0.0588    &  0.0573   &   94.5\%  &   0.2039      &    0.0955    &   0.0953     &    95.4\%    \\
		& IPW &             0.2002             &  0.0722   &   0.0673  &    94.2\% &    0.2011     &    0.1342    &    0.1281    &    94.0\%    \\
		& MLE    &         0.1993                 &   0.0472  &  0.0468   &  94.8\%   &   0.1968      &    0.0505    &    0.0500    &     94.8\%   \\
		mis       & SPIRE     &          0.2010                &   0.0588  &   0.0573  &    94.2\% &    0.2039     &    0.0955    &     0.0956   &     95.5\%   \\
		& CC &        0.2010                  &   0.0588  &  0.0573   &  94.5\%   &    0.2039     &    0.0955    &    0.0953    &   95.4\%     \\
		& IPW &            0.1995              &   0.0926  &   0.0871  &  94.8\%   &     0.2010    &    0.1474    &  0.1386      &   93.5\%     \\
		& MLE    &          0.1710                &  0.0416   &    0.0413 &   89.2\%  &     0.1197    &  0.0333      &    0.0328    &    31.3\%    \\ \hline
		&               & \multicolumn{4}{c}{70\% censoring}         & \multicolumn{4}{c}{80\% censoring} \\ \cline{3-10} 
		$f^*_{X\mid C,Z}$ & Estimator     & Mean                     & ESE & ASE & Cov & Mean    & ESE    & ASE    & Cov    \\ \hline
		true          & SPIRE      &       0.2010           &  0.1267   &  0.1263   &   95.4\%  &    0.1970     &   0.1516     &   0.1637     &   96.1\%     \\
		& CC &          0.2055                &    0.1327 &   0.1329  &  95.4\%   &     0.2097    &    0.1715    &  0.1697      &   94.9\%     \\
		& IPW &         0.1947                 &   0.1880  &  0.1802   &  94.5\%   &    0.2032     &    0.2567    &    0.2446    &    92.8\%    \\
		& MLE    &          0.1999                &   0.0546  & 0.0555    &   95.5\%  &     0.2041    &     0.0613   &   0.0617     &   95.3\%     \\
		mis       & SPIRE      &           0.2055               &   0.1327  & 0.1336    &  95.6\%   &    0.2097     &     0.1715   &   0.1684     &    94.7\%    \\
		& CC &         0.2055                 &  0.1327   &  0.1329   & 95.4\%    & 0.2097        &    0.1715    &     0.1697   &    94.9\%    \\
		& IPW &       0.1937                   &  0.1903   &  0.1798   &  93.2\%   &    0.2064     &   0.2567     &     0.2422   &   93.2\%     \\
		& MLE    &        0.1191                  &   0.0344  &   0.0353  &   36.9\%  &    0.1251     &     0.0396   &   0.0397     &    52.0\%    \\ \hline
	\end{tabular}
	\caption{Simulation results of $\beta_1$   in the controlled setting based on $N=1,000$          replicates. Mean: Average of the parameter
          estimates; ESE: the empirical standard deviation of the 
          parameter estimate; ASE: the average estimated standard
          deviation; Cov: the empirical coverage of the 95\% confidence
          interval. true: the working model $f^*_{X\mid C,Z}$ is the
          true model. mis: the working model $f^*_{X\mid C,Z}$ is the
  misspecified model. SPIRE: semiparametric informative right-censored
  covariate estimator. CC: complete case
  estimator. IPW: inverse probability weighting estimator. MLE:
  maximum likelihood estimator.
        }
		\label{table:b1s1}
\end{table}

\begin{table}[!t]
	\renewcommand{\arraystretch}{1.5}
	\begin{tabular}{ccllllllll}
		\hline
		&               & \multicolumn{4}{c}{$\beta_1$}         & \multicolumn{4}{c}{$\beta_4$} \\ \cline{3-10} 
		$f^*_{X\mid C,Z}$& Estimator     & \multicolumn{1}{c}{Mean} & ESE & ASE & Cov & Mean    & ESE    & ASE    & Cov    \\ \hline
		tru          & SPIRE      &          -1.7978               &    0.2616 &  0.2740   &   94.7\%  &  0.1953       &   0.4064     &    0.4430    &    94.4\%    \\
		& CC &           -1.7988              & 0.2616   &  0.2738   &   94.4\%  &   0.1985      &    0.4094    &   0.4411     &    94.2\%    \\
		& IPW &             -1.7908              &   0.3105  &  0.3021   &  94.0\%   &   0.1861      &    0.5263    &    0.5086    &     93.5\%   \\
		& MLE    &         -1.7983              &   0.1623  &  0.1627  &  94.4\%   &   0.1978      &    0.2420    &    0.2407    &     95.2\%   \\
		unif       &SPIRE      &          -1.8096      &   0.2618  &   0.2753  &    94.8\% &    0.2086     &    0.4088    &     0.4420   &     94.0\%   \\
		& CC &         -1.7988              & 0.2616   &  0.2738   &   94.4\%  &   0.1985      &    0.4094    &   0.4411     &    94.2\%   \\
		& IPW &            -1.8004         &   0.3493  &   0.3752  &  95.9\%   &     0.2057    &    0.5198    &  0.5310      &   94.3\%     \\
		& MLE    &          -1.7222      &  0.1461   &    0.1465 &   91.9\%  &     0.3734    &  0.2023      &    0.2032   &    86.2\%    \\ 
		K-M       &SPIRE      &          -1.8089      &   0.2602  &   0.2731  &    94.8\% &    0.2079     &    0.4085    &     0.4404   &     94.0\%   \\
		& CC &         -1.7988              & 0.2616   &  0.2738   &   94.4\%  &   0.1985      &    0.4094    &   0.4411     &    94.2\%   \\
		& IPW &            -1.8004         &   0.3493  &   0.3752  &  95.9\%   &     0.2057    &    0.5198    &  0.5310      &   94.3\%     \\
		& MLE    &          -1.7237      &  0.1999  &    0.2122 &   95.4\%  &     0.3004    &  0.2816      &    0.2793   &    93.3\%    \\ \hline
	\end{tabular}
	\caption{Simulation results of $\beta_1$ and $\beta_4$ in the realistic setting based on $N=1,000
$		replicates. tru: the working model $f^*_{X\mid C,Z}$ is the
		true model. unif: the working model $f^*_{X\mid C,Z}$ is the
		uniform model.  K-M: the working model $f^*_{X\mid C,Z}$ is the
		localized Kaplan-Meier estimator.  Mean, ESE, ASE, and Cov as in Table \ref{table:b1s1}.        
	}
	\label{table:b1s2}
\end{table}

\subsection{Evaluation of power to detect differences between
  noninformative and informative covariate censoring}

We next evaluated whether the chi-square test (Theorem \ref{thm:chisq})
can correctly identify when covariate censoring is informative versus
noninformative. We modified the controlled setting by introducing a
dependency parameter $\alpha$ to modulate the relationship between $C$
and $X$ given $Z$. With $N=1,000$ samples of $n=3,500$ observations
each, we generated $Z \sim \text{Bernoulli}(0.5)$, $C|Z \sim
\text{Uniform}(Z-1, Z+1)$, and $X|C,Z \sim \text{Normal}\{\alpha C +
\mu, (Z+1)/\sigma^2\}$, with $Y$ and $\bb$ as in the controlled
setting. We varied $\alpha$, $\mu$, and $\sigma$ to generate different
dependency levels while maintaining 80\% right-censoring: $\alpha=0$
produces noninformative covariate censoring ($X \indep C|Z$), while
$\alpha > 0$ produces informative covariate censoring. To quantify the
conditional dependence between $C$ and $X$ given $Z$, we used the
conditional dependence coefficient proposed by
\cite{azadkia2021simple}. All estimators used the working model
$f_{X|C,Z}^* = f_{X|Z}$, which ignores dependence on $C$. This
specification is correct under noninformative covariate censoring but
incorrect under informative covariate censoring.  
We computed the test statistic $T_{\text{chi}} = n(\wh{\bb}_1 -
\wh{\bb}_2)^{\trans} \wh{V}^{-1}(\wh{\bb}_1 - \wh{\bb}_2)$, where
$\wh{\bb}_1$ is the CC estimator, the IPW estimator, or SPIRE and
$\wh{\bb}_2$ is the MLE, rejecting the null hypothesis of
noninformative covariate censoring at 5\% significance level
when $T_{\text{chi}} > \chi^2_{3,0.05} = 7.81$, where 3 equals the
dimension of $\bb$.

We evaluated both empirical size---the test's ability to maintain the nominal 5\% level---and empirical power---its ability to detect informative covariate censoring. While all three tests are asymptotically valid under the null hypothesis, finite-sample performance varied: SPIRE achieved an empirical size of 0.049, the CC estimator was slightly conservative (0.035), and the IPW estimator was slightly liberal (0.076). These differences, though modest, reflect finite-sample variability rather than theoretical distinctions. Figure \ref{fig:indep.test} shows empirical power across dependency levels. SPIRE and the CC estimator achieved similar power at all dependency levels, with both having sufficient efficiency to detect informative covariate censoring, whereas the IPW estimator's higher variance limited its power to detect departures from the null.

We validated these findings using the realistic setting, where the data generation has $C$ depend on $X$ given $\Z$. The chi-square test correctly identified this informative covariate censoring with empirical power of 0.967 (SPIRE), 0.998 (CC),  and 0.821 (IPW). These high power values show that the test can reliably detect informative covariate censoring when it exists.

\begin{figure}[H]
	\centering
	\includegraphics[width=0.6\linewidth]{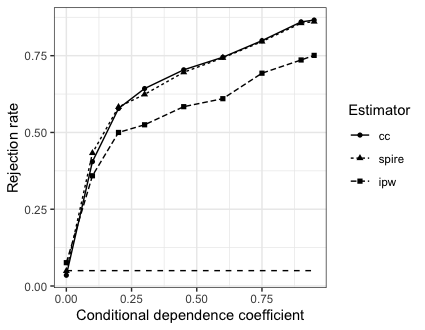}
	\caption{Simulation results of the empirical power and the
          empirical size based on 1000 replicates. ${\rm cc}$: when $\wh
          \bb_1$ is the CC estimator; ${\rm spire}$: when
          $\wh \bb_1$ is the SPIRE estimator; ${\rm ipw}$: when $\wh
          \bb_1$ is the IPW
          estimator. The horizontal dashed line represents the 0.05
          significance level.} 
	\label{fig:indep.test}
	\end{figure}

\section{Application to Enroll-HD data}
\label{sec:app}

Huntington disease offers a unique opportunity to study
neurodegenerative progression: a single, fully penetrant genetic
mutation allows definitive identification of future patients decades
before diagnosis \citep{scahill2020biological}. Unlike Alzheimer or
Parkinson disease, where at-risk populations remain uncertain,
individuals carrying the Huntington disease mutation can be followed
from health through decline, revealing when interventions might be
most effective \citep{langbehn2019association}. 

We analyzed data from 3,657 mutation carriers in Enroll-HD
\citep{sathe2021enroll}, a large, observational study of Huntington disease; all had entered the study without a diagnosis. 
All carried expanded CAG repeats ($\geq$40 repeats), the genetic mutation that
causes Huntington disease, with complete penetrance. Diagnosis occurred
when a clinician reached definite confidence that motor signs
represented disease manifestation, recorded as a diagnosis confidence
level (DCL) of 4 on a scale from 1 (low confidence) to 4 (definite)
\citep{hogarth2005interrater}. With 84.7\% of participants exiting
before diagnosis, their time of diagnosis $X$ was right-censored at
time of study exit $C$. Our clinical measure $Y$ was
the composite score from the Unified Huntington Disease Rating Scale (cUHDRS),
which
integrates motor, cognitive, and functional assessments; higher scores
indicate worse impairment \citep{schobel2017motor}.  
We modeled:
\bse
Y \sim \text{Normal}\{\beta_0 + \beta_1(X-Z_{\text{age}_0}) + \beta_2 Z_{\text{CAG}} + \beta_3 Z_{\text{sex}} + \beta_4(X-Z_{\text{age}_0})Z_{\text{sex}}, \sigma^2\},
\ese
where $X-Z_{\text{age}_0}$ anchors patients by years from study entry
to diagnosis (with $Z_{\text{age}_0}$ denoting age at study entry);
$Z_{\text{CAG}}$ is CAG repeat length; and $Z_{\text{sex}}$ indicates
female sex. We transformed each of the quantities
$X-Z_{\text{age}_0}$, $C-Z_{\text{age}_0}$, and $Z_{\text{CAG}}$ to
the $(0,1)$ interval (subtracting the minimum and dividing by the
range within each), allowing us to implement working models for
$f_{X|C,\mathbf{Z}}$ using standard distributions over $(0,1)$. 

To test for noninformative covariate censoring, we applied the
  localized Kaplan-Meier estimator
  $\widehat{S}_{X-Z_{\text{age}_0}|\mathbf{Z}}$ in \eqref{eq:KM}
  (bandwidth $h=0.20$) and obtained its derivative as our working model
  for the density
  $f_{X-Z_{\text{age}_0}|C-Z_{\text{age}_0},\mathbf{Z}}^*$. This
  working model assumes independence between time of diagnosis and
  time of study exit, precisely the assumption being tested. The test
  statistics comparing the CC estimator, the IPW estimator, and SPIRE
  against the MLE were 79.70, 62.44, and 74.34, respectively, all with
  p-values $<0.0001$, rejecting noninformative covariate censoring.  

Given this evidence of informative covariate censoring,  we next examined estimator performance under model misspecification. We applied  all four estimators using two deliberately misspecified working models for $f_{X-Z_{\text{age}_0}|C-Z_{\text{age}_0},\mathbf{Z}}^*$: (1) the uniform distribution over $(0,1)$, which ignores all covariate dependencies, and (2) the localized Kaplan-Meier estimator from our test, which incorrectly assumes independence. Figure \ref{fig:real_ci} presents 95\% confidence intervals for all parameters under both working models. All estimators show $\beta_1 < 0$, indicating that cUHDRS scores deteriorate as patients approach diagnosis, as expected in a progressive neurodegenerative disease.
SPIRE and the CC estimator produce similar estimates for $\beta_1$,
while the IPW estimator and MLE yield attenuated estimates, with the
MLE closest to zero. This attenuation could underestimate how rapidly
pre-diagnosis decline occurs, potentially leading researchers to
conclude that the intervention window is wider than it actually
is. The MLE's narrow confidence intervals compound this problem by
lending false certainty to the underestimate under informative
covariate censoring. Such misestimation could misdirect therapeutic
development by suggesting more time exists to detect treatment effects
than patients actually have before irreversible damage occurs. 

SPIRE's maintained estimation of $\beta_1=-0.9$ despite two forms of
misspecification---ignoring all dependencies or incorrectly assuming
independence---demonstrates its robustness for quantifying how rapidly
decline occurs when the true censoring mechanism remains unknown. This
robustness matters: accurately capturing how rapid pre-diagnosis
decline is directly informs how long trials must run to detect
treatment effects and how quickly patients approach irreversible
damage. In studies where 85\% right-censoring is common and dropout
patterns cannot be verified, SPIRE provides the consistency needed to
reliably find intervention windows.

\begin{figure}[H]
  \centering
  \includegraphics[width=0.45\textwidth]{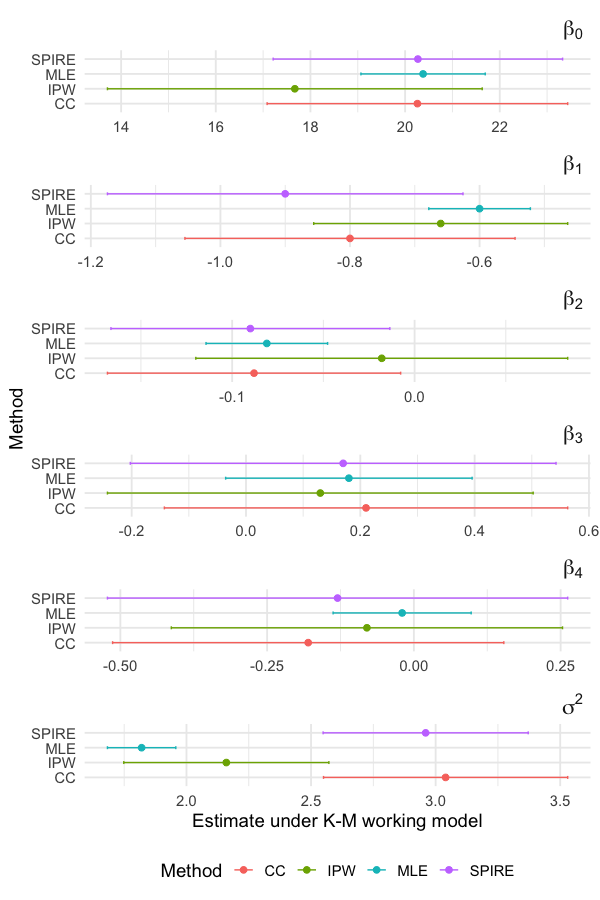}
  \hfill
  \includegraphics[width=0.45\textwidth]{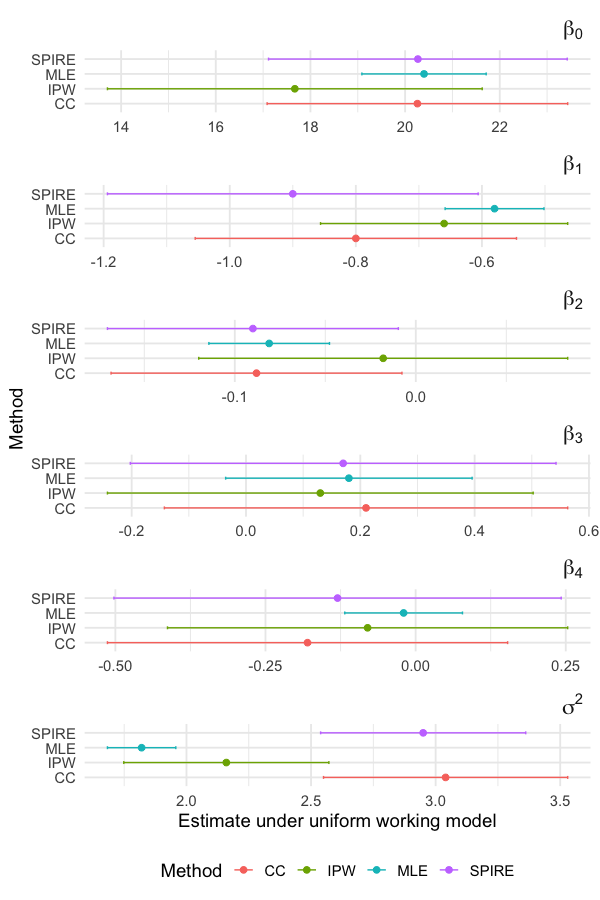}
  \caption{$95\%$ Confidence intervals for all parameters where the working model is the localized Kaplan-Meier (K-M) estimator (left) and the uniform model (right).} 
  \label{fig:real_ci}
\end{figure}

\section{Discussion}
\label{sec:discussion}

Our work shows that handling informative covariate censoring is
  more tractable than previously believed. The field has treated the
  joint density $f_{C,X|\Z}$ as fundamentally unverifiable because $C$
  and $X$ are never observed together. SPIRE  recognizes but reframes this challenge
    by being an estimator that does not rely on specifying 
      $f_{C,X|\Z}$ correctly.
When we decompose $f_{C,X|\Z} = f_{X|C,\Z} 
  f_{C|\Z}$  and derive the efficient score, the $f_{C|\Z}$ term
  cancels through the construction, and we bypass the specification 
  of $f_{X\mid C,\Z}$ through orthogonisation.
  This procedure shifts the
  paradigm from specifying unobservable densities to constructing
  estimators that circumvent them entirely. 

The chi-square test we developed complements SPIRE by transforming the
untestable assumption of noninformative covariate censoring into a
testable hypothesis. By exploiting the differential consistency between
estimators under misspecification, researchers can now determine
whether informative covariate censoring affects their data. Studies
examining how clinical measures change as patients approach diagnosis
have typically assumed noninformative covariate censoring
\citep{kong2018conditional,
  chu2020stochastic,scahill2020biological}---not because of oversight, but
because tools to handle or test informative covariate censoring were
unavailable. SPIRE and the accompanying test fill these gaps.

The immediate impact is practical: researchers analyzing
right-censored covariates no longer face the robustness-efficiency
trade-off that has characterized the censored covariate field. They can test for
informative covariate censoring, apply SPIRE if detected, and obtain
consistent,  potentially efficient estimates regardless of
modeling assumptions. The discrete approximation we demonstrated makes
implementation straightforward, while the sandwich variance formula
provides valid inference. For fields where 80--90\% right-censoring is common and the type of covariate censoring remains unknown, SPIRE and the accompanying test allow researchers to reliably estimate how rapidly decline occurs and identify intervention windows without relying on assumptions they cannot verify---moving beyond the constraints that have limited our understanding of pre-diagnosis decline.

\bibliographystyle{agsm}
\bibliography{main}

\section*{Supplementary Material}

\setcounter{equation}{0}\renewcommand{\theequation}{S.\arabic{equation}}
\setcounter{subsection}{0}\renewcommand{\thesubsection}{S.\arabic{subsection}}
\setcounter{table}{0}\renewcommand{\thetable}{S.\arabic{table}}

\subsection{Proof of Lemma \ref{lem:1}}\label{proofoflem1}

To prove identifiability, we proceed by contradiction. Suppose there exist two distinct parameters $\bb$ and $\wt\bb$, with their respective associated nuisance parameters $f$ and $\wt f$, that yield the same likelihood for any single observation. Then
\bse
&&\left\{\int_{x}^{\infty}f_{Y|X,\Z}(y,x,\z;\bb)f_{C,X|\Z}(c,x,\z)dc\right\}^{\delta}\\
&&\left\{\int_{c}^{\infty}f_{Y|X,\Z}(y,x,\z;\bb)f_{C,X|\Z}(c,x,\z)dx\right\}^{1-\delta}f_{\Z}(\z)\\
&=&\left\{\int_{x}^{\infty}f_{Y|X,\Z}(y,x,\z;\wt\bb)\wt
  f_{C,X|\Z}(c,x,\z)dc\right\}^{\delta}\\
  &&\left\{\int_{c}^{\infty}f_{Y|X,\Z}(y,x,\z;\wt\bb)\wt
  f_{C,X|\Z}(c,x,\z)dx\right\}^{1-\delta}\wt f_{\Z}(\z).
\ese
Substituting $\delta=1$ and $\delta=0$ into the above equation separately gives us two distinct equations:
\bse
\int_{x}^{\infty}f_{Y|X,\Z}(y,x,\z;\bb)f_{C,X|\Z}(c,x,\z)dcf_{\Z}(\z)
&=&\int_{x}^{\infty}f_{Y|X,\Z}(y,x,\z;\wt\bb)\wt
  f_{C,X|\Z}(c,x,\z)dc\wt f_{\Z}(\z),\n\\
\int_{c}^{\infty}f_{Y|X,\Z}(y,x,\z;\bb)f_{C,X|\Z}(c,x,\z)dxf_{\Z}(\z)
&=&\int_{c}^{\infty}f_{Y|X,\Z}(y,x,\z;\wt\bb)\wt
  f_{C,X|\Z}(c,x,\z)dx\wt f_{\Z}(\z).
\ese
 Integrating in $y$ leads to
\bse
\int_{x}^{\infty}f_{C,X|\Z}(c,x,\z)dcf_{\Z}(\z)
&=&\int_{x}^{\infty}\wt
  f_{C,X|\Z}(c,x,\z)dc\wt f_{\Z}(\z),\\
\int_{c}^{\infty}f_{C,X|\Z}(c,x,\z)dxf_{\Z}(\z)
&=&\int_{c}^{\infty}\wt
  f_{C,X|\Z}(c,x,\z)dx\wt f_{\Z}(\z),
\ese
which leads to
\bse
\iint_{t<x<c}f_{C,X|\Z}(c,x,\z)dcdxf_{\Z}(\z)
&=&\iint_{t<x<c}\wt
  f_{C,X|\Z}(c,x,\z)dcdx\wt f_{\Z}(\z),\\
\iint_{t<c<x}f_{C,X|\Z}(c,x,\z)dxdcf_{\Z}(\z)
&=&\iint_{t<c<x}\wt
  f_{C,X|\Z}(c,x,\z)dxdc\wt f_{\Z}(\z).
\ese
Taking the sum and letting $t=-\infty$, we get
$f_{\Z}(\z)=\wt f_{\Z}(\z)$, and subsequently,
\be
\int_x^{\infty}f_{C,X|\Z}(c,x,\z)dc
&=&\int_x^{\infty}\wt
  f_{C,X|\Z}(c,x,\z)dc,\n\\
\int_{c}^{\infty}f_{C,X|\Z}(c,x,\z)dx
&=&\int_{c}^{\infty}\wt
  f_{C,X|\Z}(c,x,\z)dx,\n\\
\int_t^\infty \int_t^\infty
f_{C,X|\Z}(c,x,\z)dcdx
&=&\int_t^\infty \int_t^\infty
 \wt f_{C,X|\Z}(c,x,\z)dcdx,\n\\
\int_{c}^{\infty}f_{Y|X,\Z}(y,x,\z;\bb)f_{C,X|\Z}(c,x,\z)dx
&=&\int_{c}^{\infty}f_{Y|X,\Z}(y,x,\z;\wt\bb)\wt
  f_{C,X|\Z}(c,x,\z)dx.\n\\
  \label{eq:a4}
\ee
The first relation above can be equivalently written as
\be\label{eq:use1}
f_{X|\Z}(x,\z)S_{C|X,Z}(x,x,\z)=\wt f_{X|\Z}(x,\z)\wt S_{C|X,Z}(x,x,\z).
\ee

Alternatively, we can rewrite the likelihood as
\bse
&&\left\{\int_{x}^{\infty}f_{Y|\Z}(y,\z)f_{X|Y,\Z}(x,y,\z)
f_{C|X,\Z}(c,x,\z)
dc\right\}^{\delta}\\
&&\left\{\int_{c}^{\infty}
f_{Y|\Z}(y,\z)f_{X|Y,\Z}(x,y,\z)f_{C|X,\Z}(c,x,\z)dx\right\}^{1-\delta}\n\\
&=&\left\{\int_{x}^{\infty}\wt f_{Y|\Z}(y,\z)\wt f_{X|Y,\Z}(x,y,\z)
\wt f_{C|X,\Z}(c,x,\z)
dc\right\}^{\delta}\\
&&\left\{\int_{c}^{\infty}
\wt f_{Y|\Z}(y,\z)\wt f_{X|Y,\Z}(x,y,\z)\wt f_{C|X,\Z}(c,x,\z)dx\right\}^{1-\delta},
\ese
hence
\bse
&&\int_{x}^{\infty}f_{Y|\Z}(y,\z)f_{X|Y,\Z}(x,y,\z)
f_{C|X,\Z}(c,x,\z)
dc\\
&=&\int_{x}^{\infty}\wt f_{Y|\Z}(y,\z)\wt f_{X|Y,\Z}(x,y,\z)
\wt f_{C|X,\Z}(c,x,\z)dc,\n\\
&&\int_{c}^{\infty}
f_{Y|\Z}(y,\z)f_{X|Y,\Z}(x,y,\z)f_{C|X,\Z}(c,x,\z)dx\\
&=&\int_{c}^{\infty}
\wt f_{Y|\Z}(y,\z)\wt f_{X|Y,\Z}(x,y,\z)\wt f_{C|X,\Z}(c,x,\z)dx.
\ese
This result leads to
\bse
&&\iint_{x<c}f_{Y|\Z}(y,\z)f_{X|Y,\Z}(x,y,\z)
f_{C|X,\Z}(c,x,\z)
dcdx\\
&=&\iint_{x<c}\wt f_{Y|\Z}(y,\z)\wt f_{X|Y,\Z}(x,y,\z)
\wt f_{C|X,\Z}(c,x,\z)dcdx,\n\\
&&\iint_{c<x}
f_{Y|\Z}(y,\z)f_{X|Y,\Z}(x,y,\z)f_{C|X,\Z}(c,x,\z)dxdc\\
&=&\iint_{c<x}
\wt f_{Y|\Z}(y,\z)\wt f_{X|Y,\Z}(x,y,\z)\wt f_{C|X,\Z}(c,x,\z)dxdc.
\ese
Adding these two equations together gives
\bse
&&\iint f_{Y|\Z}(y,\z)f_{X|Y,\Z}(x,y,\z)
f_{C|X,\Z}(c,x,\z)
dcdx\\
&=&\iint \wt f_{Y|\Z}(y,\z)\wt f_{X|Y,\Z}(x,y,\z)
\wt f_{C|X,\Z}(c,x,\z)dcdx,
\ese
i.e., $f_{Y|\Z}(y,\z)=\wt f_{Y|\Z}(y,\z)$. This subsequently leads
to
\bse
\int_{x}^{\infty}f_{X|Y,\Z}(x,y,\z)
f_{C|X,\Z}(c,x,\z)
dc
&=&\int_{x}^{\infty}\wt f_{X|Y,\Z}(x,y,\z)
\wt f_{C|X,\Z}(c,x,\z)dc,\n\\
\int_{c}^{\infty}
f_{X|Y,\Z}(x,y,\z)f_{C|X,\Z}(c,x,\z)dx&=&\int_{c}^{\infty}
\wt f_{X|Y,\Z}(x,y,\z)\wt f_{C|X,\Z}(c,x,\z)dx.
\ese
The first relation above can be written as
\be\label{eq:use2}
f_{X|Y,\Z}(x,y,\z)S_{C|X,Z}(x,x,\z)=\wt f_{X|Y,\Z}(x,y,\z)\wt S_{C|X,Z}(x,x,\z).
\ee
Taking the ratio of \eqref{eq:use2} and \eqref{eq:use1}, we get
\be\label{eq:ratio}
\frac{f_{X|Y,\Z}(x,y,\z)}{f_{X|\Z}(x,\z)}
=\frac{\wt f_{X|Y,\Z}(x,y,\z)}{\wt f_{X|\Z}(x,\z)}, 
\ee
which further leads to
\bse
f_{Y\mid X,\Z}(y,x,\z,\bb)
&=&
\frac{f_{X|Y,\Z}(x,y,\z)f_{Y|\Z}(y,\z)}{f_{X|\Z}(x,\z)}\\
&=&\frac{\wt f_{X|Y,\Z}(x,y,\z) f_{Y|\Z}(y,\z)}{\wt f_{X|\Z}(x,\z)}\\
&=&f_{Y\mid X,\Z}(y,x,\z,\wt\bb).
\ese
Hence, $\bb=\wt\bb$, i.e., $\bb$ is identifiable.

Now $f_{Y|X,\Z}(y,x,\z,\bb)$ and $f_{Y|\Z}(y,\z)$ are both unique, and
\bse
f_{Y|\Z}(y,\z)&=&\int f_{Y|X,\Z}(y,x,\z,\bb)f_{X|\Z}(x,\z)dx\\
&=&\int f_{Y|X,\Z}(y,x,\z,\bb)\wt f_{X|\Z}(x,\z)dx.
\ese
Under the completeness condition, we get $f_{X|\Z}(x,\z)=\wt
f_{X|\Z}(x,\z)$. This result together with
\eqref{eq:use1} 
leads to $S_{C|X,\Z}(x,x,\z)=\wt S_{C|X,\Z}(x,x,\z)$.
Similarly, \eqref{eq:a4} leads
to
\bse
&&\int I(c<x)f_{Y|X,\Z}(y,x,\z;\bb)f_{C|X,\Z}(c,x,\z)f_{X|Z}(x,\z)dx\\
&=&\int_{c}^{\infty}f_{Y|X,\Z}(y,x,\z;\bb)f_{C,X|\Z}(c,x,\z)dx\\
&=&\int_{c}^{\infty}f_{Y|X,\Z}(y,x,\z;\bb)\wt
  f_{C,X|\Z}(c,x,\z)dx\\
&=&\int I(c<x)f_{Y|X,\Z}(y,x,\z;\bb)\wt f_{C|X,\Z}(c,x,\z) f_{X|Z}(x,\z)
 dx.
\ese
Hence, the completeness condition leads to
\bse
I(c<x) f_{C|X,\Z}(c,x,\z) f_{X|Z}(x,\z)
=I(c<x)\wt f_{C|X,\Z}(c,x,\z) f_{X|Z}(x,\z),
\ese
and, thus, $f_{C|X,\Z}(c,x,\z)=\wt f_{C|X,\Z}(c,x,\z)$ for all $c<x$.
Now the likelihood of a single observation can be written as 
\bse
&&\left\{f_{Y|X,\Z}(y,x,\z;\bb)S_{C,X|\Z}(x,x,\z)\right\}^{\delta}\\
&&\left\{\int
  f_{Y|X,\Z}(y,x,\z;\bb)f_{C|X,\Z}(c,x,\z)
I(c<x)f_{X|\Z}(x,\z)
dx\right\}^{1-\delta}f_\Z(\z).
\ese
We have therefore proven that  $f_{C,X|\Z}$, $f_{\Z}$, $\bb$, are
identifiable. 
\qed

\subsection{Specific form of the nuisance tangent space $\Lambda$ and its proof}\label{proofofpro1}

\begin{Pro}\label{pro:Lambda}
The nuisance tangent space is $\Lambda \equiv
\Lambda_m\oplus\Lambda_\z$, where $\Lambda_m$ and $\Lambda_\z$ are the
nuisance tangent spaces for $f_{C,X|\Z}(c,x,\z)$ and $f_\Z(\z)$,
respectively. Here, $\Lambda_m$ stands for the main nuisance tangent space. Specifically,
\bse
\Lambda_m&=&\left[\delta 
\frac{E\{\a_1(C,x,\z)I(x\le C)\mid x,\z\}}{E\{I(x\le C)\mid x,\z\}}
+(1-\delta) \frac{E\{\a_1(c,X,\z)I(X> c)\mid y, c,\z\}}{E\{I(X>
  c)\mid y, c,\z\}}:\right.\\
&&\left.E\{\a_1(C,X,\z)\mid\z\}=\0
\right],\\
\Lambda_\z&=&[\a_2(\z): E\{\a_2(\Z)\}=\0].
\ese
\end{Pro}

\begin{proof}
From \eqref{eq:model}, it is straightforward to derive that the nuisance scores associated with $f_{C,X|\Z}, f_\Z$, denoted respectively as $\S_1,\S_2$, are
\bse
\S_1(y,w,\delta,\z)&=&\delta \frac{\int_x^\infty\a_1(c,x,\z)f_{C,X|\Z}(c,x,\z)dc}{\int_x^\infty f_{C,X|\Z}(c,x,\z)dc}\\
&&+(1-\delta)\frac{
\int_{c}^{\infty}\a_1(c,x,\z)f_{Y|X,\Z}(y,x,\z;\bb)f_{C,X|\Z}(c,x,\z)dx
}{\int_{c}^{\infty}f_{Y|X,\Z}(y,x,\z;\bb)f_{C,X|\Z}(c,x,\z)dx}\\
&=&\delta 
\frac{E\{\a_1(C,x,\z)I(x\le C)\mid x,\z\}}{E\{I(x\le C)\mid x,\z\}}\\
&&+(1-\delta) \frac{E\{\a_1(c,X,\z)I(X> c)\mid y, c,\z\}}{E\{I(X>
  c)\mid y, c,\z\}}, \\
\S_2(y,w,\delta,\z)&=&\a_2(\z),
\ese
where $\a_1(c,x,\z),\a_2(\z)$ satisfy $E\{\a_1(C,X,\z)\mid\z\}=\0, E\{\a_2(\z)\}=\0$, respectively. 
The nuisance tangent spaces associated with $f_{C,X|\Z}$ and $f_\Z$ can now be identified as $\Lambda_m$ and $\Lambda_\z$, respectively (as defined in Proposition \ref{pro:Lambda}), since these spaces are formed by the linear spans of their corresponding nuisance scores.  Next, we show that $\Lambda_m \perp \Lambda_\z$. For any element in $\Lambda_m$, we have
\bse
&&E\left[
\delta\frac{\int_x^\infty\a_1(c,x,\z)f_{C,X|\Z}(c,x,\z)dc}{\int_x^\infty f_{C,X|\Z}(c,x,\z)dc}\right.\\
&&\left. +(1-\delta)\frac{
\int_{c}^{\infty}\a_1(c,x,\z)f_{Y|X,\Z}(y,x,\z;\bb)f_{C,X|\Z}(c,x,\z)dx
}{\int_{c}^{\infty}f_{Y|X,\Z}(y,x,\z;\bb)f_{C,X|\Z}(c,x,\z)dx}\mid\z
\right]\\
&=&\int 
\left\{
\frac{\int_x^\infty\a_1(c,x,\z)f_{C,X|\Z}(c,x,\z)dc}{\int_x^\infty f_{C,X|\Z}(c,x,\z)dc}\right\}
\left\{f_{Y|X,\Z}(y,x,\z;\bb)\int_x^{\infty}f_{C,X|\Z}(c,x,\z)dc\right\}dxdy\\
&&+\int 
\left\{
\frac{
\int_{c}^{\infty}\a_1(c,x,\z)f_{Y|X,\Z}(y,x,\z;\bb)f_{C,X|\Z}(c,x,\z)dx
}{\int_{c}^{\infty}f_{Y|X,\Z}(y,x,\z;\bb)f_{C,X|\Z}(c,x,\z)dx}\right\}\\
&&\left\{\int_c^{\infty}f_{Y|X,\Z}(y,x,\z;\bb)f_{C,X|\Z}(c,x,\z)dx\right\}
dcdy\\
&=&\int 
\int_x^\infty\a_1(c,x,\z)f_{C,X|\Z}(c,x,\z)dc dx
+\int 
\int_{c}^{\infty}\a_1(c,x,\z)f_{C,X|\Z}(c,x,\z)dx
dc\\
&=&E\{\a_1(C,X,\z)\mid\z\}=\0,
\ese
so $\Lambda_m\perp\Lambda_\z$.
\end{proof}

\subsection{Proof of Proposition \ref{pro:LambdaP}}\label{proofofpro2}

Since $\Lambda$ is the sum of $\Lambda_m$ and $\Lambda_\z$, it follows that $\Lambda^\perp=\Lambda_m^\perp\cap \Lambda_\z^\perp$. Also, $\Lambda_\z^\perp=[\bfb(y,w,\delta,\z):
E\{\bfb(Y,W,\Delta,\z)\mid\z\}=\0]$. \\
Let the set $\bA\equiv \left[\bfb(y,w,\delta,\z)=\delta \bfb_1(y,x,\z)+(1-\delta)
\bfb_0(y,c,\z):
E\{\bfb(Y,w,\Delta,\z) \mid c, x,\z\}=\0\right]$.

For any $\bfb(y,w,\delta,\z) \in \bA$:
\bse
&&E\left[\bfb\trans(Y,W,\delta,\Z)
\delta \frac{\int_X^\infty\a(c,X,\Z)f_{C,X|\Z}(c,X,\Z)dc}{\int_X^\infty f_{C,X|\Z}(c,X,\Z)dc}\right.\\
&&\left.+\bfb\trans(Y,W,\delta,\Z)(1-\delta)\frac{
	\int_{C}^{\infty}\a(C,x,\Z)f_{Y|X,\Z}(Y,x,\Z;\bb)f_{C,X|\Z}(C,x,\Z)dx
}{\int_{C}^{\infty}f_{Y|X,\Z}(Y,x,\Z;\bb)f_{C,X|\Z}(C,x,\Z)dx}\right]\\
&=&E\left[\int \bfb_1\trans(y,x,\Z)
\frac{\int_x^\infty\a(c,x,\Z)f_{C,X|\Z}(c,x,\Z)dc}{\int_x^\infty f_{C,X|\Z}(c,x,\Z)dc}
f_{Y|X,\Z}(y,x,\Z;\bb)\right.\\
&&\left.\int_x^{\infty}f_{C,X|\Z}(c,x,\Z)dcdxdy\right.\\
&&+\int\bfb_0\trans(y,c,\Z) \frac{
	\int_{c}^{\infty}\a(c,x,\Z)f_{Y|X,\Z}(y,x,\Z;\bb)f_{C,X|\Z}(c,x,\Z)dx
}{\int_{c}^{\infty}f_{Y|X,\Z}(y,x,\Z;\bb)f_{C,X|\Z}(c,x,\Z)dx}\\
&&\left.\int_c^{\infty}f_{Y|X,\Z}(y,x,\Z;\bb)f_{C,X|\Z}(c,x,\Z)dx
dcdy\right]\\
&=&E\left[\int \{\bfb_1(y,x,\Z) I(x\le c)+\bfb_0(y,c,\Z) I(x>c)\}\trans\right.\\
&&\left.\a(c,x,\Z)f_{C,X|\Z}(c,x,\Z)f_{Y|X,\Z}(y,x,\Z;\bb)dc
dxdy\right]\\
&=&E[E\{\bfb_1(Y,X,\Z) I(X\le C)+\bfb_0(Y,C,\Z) I(X>C)\mid C,
X,\Z\}\trans\a(C,X,\Z)]\\
&=&E[E\{\bfb(Y,W,\delta,\Z) \mid C,
X,\Z\}\trans\a(C,X,\Z)] \\
&=&0
\ese
for any $\a(c,x,\z)$ described in $\Lambda_m$, which satisfies $E\{\a(C,X,\z)|\z\}=\0$. 

Thus, $\bA \subset \Lambda_m^\perp$. In addition, $\bA\subset \Lambda_\z^\perp$ since each element of $\bA$ satisfies the orthogonality condition with respect to $\Lambda_\z$.  Then we have $\bA\subset \Lambda^\perp$. 

Conversely, for any $\bfb(y,w,\delta,\z) \in \Lambda^\perp$ (it is
also in $\Lambda_m^\perp$), we have: 
\bse
E[E\{\bfb(Y,W,\delta,\Z) \mid C,
X,\Z\}\trans\a(C,X,\Z)] = 0
\ese
for any $\a(c,x,\z)$ which satisfies $E\{\a(C,X,\z)|\z\}=\0$.

Take $\a(c,x,\z) = E[\bfb(Y,w,\delta,\z)|c,x,\z]$. Then, $E\{\a(C,X,\z)|\z\}=\0$ due to the fact $\bfb(y,w,\delta,\z) \in \Lambda^\perp$. Thus, we have $E[E\{\bfb(Y,W,\delta,\Z) \mid C,
X,\Z\}\trans E[\bfb(Y,W,\delta,\Z)|C,X,\Z]] = 0$, implying $E[\bfb(Y,w,\delta,\z)|c,x,\z]=0$.

Thus, we have shown that $\Lambda^\perp \subset \bA$.
To conclude, we have thus shown:
\bse
\Lambda^\perp=\left[\bfb(y,w,\delta,\z)=\delta \bfb_1(y,x,\z)+(1-\delta)
\bfb_0(y,c,\z):
 E\{\bfb(Y,w,\Delta,\z) \mid c, x,\z\}=\0\right].
\ese

\subsection{Proof of Proposition \ref{pro:Seff}}\label{proofofpro3}
The score vector $\S_\bb$ is \bse
&&\S_\bb(y,w,\delta,\z,\bb,f_{C,X|\Z})\\
&=&\delta\frac{\partial}{\partial\bb}\log f_{Y|X,\Z}(y,x,\z;\bb)+
(1-\delta)\frac{\partial}{\partial\bb}
\log
\left\{\int_c^{\infty}f_{Y|X,\Z}(y,x,\z;\bb)f_{C,X|\Z}(c,x,\z)dx\right\}\\
&=&\delta\S_\bb^F(y,x,\z;\bb)+
(1-\delta)\frac{
\int_c^{\infty}\S_\bb^F(y,x,\z;\bb) f_{Y|X,\Z}(y,x,\z;\bb)f_{C,X|\Z}(c,x,\z)dx}
{\int_c^{\infty}f_{Y|X,\Z}(y,x,\z;\bb)f_{C,X|\Z}(c,x,\z)dx},
\ese
where 
\bse
\S_\bb^F(y,x,\z;\bb)\equiv \partial \log f_{Y|X,\Z}(y,x,\z;\bb)/\partial\bb.
\ese

Using the definition of $\S_\bb(y,w,\delta,\z,\bb,f_{C,X|\Z})$ given above, we can prove that $E\{\S_\bb(Y,W,\Delta,\z,\bb,f_{C,X|\Z})\mid\z\}=\0$,  so 
$\S_\bb(y,w,\delta,\z,\bb,f_{C,X|\Z})\in\Lambda_\z^\perp$.
We write 
\be
\S_\bb(y,w,\delta,\z)
&=&\S(y,w,\delta,\z)
+\delta
\frac{\int_x^\infty\a_0(c,x,\z)f_{C,X|\Z}(c,x,\z)dc}{\int_x^\infty f_{C,X|\Z}(c,x,\z)dc}\label{eq:initial-form-seff}\\
&&+(1-\delta)\frac{
\int_{c}^{\infty}\a_0(c,x,\z)f_{Y|X,\Z}(y,x,\z;\bb)f_{C,X|\Z}(c,x,\z)dx
}{\int_{c}^{\infty}f_{Y|X,\Z}(y,x,\z;\bb)f_{C,X|\Z}(c,x,\z)dx},\nonumber
\ee
where
\bse
&&E\{\a_0(C,X,\z)\mid\z\}=\0,
\ese
and
\bse
&&E\{\S\eff(Y,w,\Delta,\z) \mid c, x,\z\}\\
&=&E\{\S_\bb(Y,w,\Delta,\z) \mid c,x,\z\}-\frac{\int_x^\infty\a_0(c,x,\z)f_{C,X|\Z}(c,x,\z)dc}{\int_x^\infty f_{C,X|\Z}(c,x,\z)dc}
I(x\le c) \\
&&-\int\frac{
\int_{c}^{\infty}\a_0(c,x,\z)f_{Y|X,\Z}(y,x,\z;\bb)f_{C,X|\Z}(c,x,\z)dx
}{\int_{c}^{\infty}f_{Y|X,\Z}(y,x,\z;\bb)f_{C,X|\Z}(c,x,\z)dx}
f_{Y|X,\Z}(y,x,\z;\bb)I(c<x) dy\\
&=&\0,
\ese
where the first equality follows from (\ref{eq:initial-form-seff}). The last equality uses our earlier result that elements of $\Lambda^{\perp}$ necessarily satisfy $E\{\S\eff(Y,w,\Delta,\z)|c,x,\z\}=\0$. This result
implies that  $\a_0$ satisfies 
\bse
&&E\{\S_\bb(Y,w,\Delta,\z) \mid c,x,\z\}\n\\
&=&\frac{\int_x^\infty\a_0(c,x,\z)f_{C,X|\Z}(c,x,\z)dc}{\int_x^\infty f_{C,X|\Z}(c,x,\z)dc}
I(x\le c) \n\\
&&+\int\frac{
\int_{c}^{\infty}\a_0(c,x,\z)f_{Y|X,\Z}(y,x,\z;\bb)f_{C,X|\Z}(c,x,\z)dx
}{\int_{c}^{\infty}f_{Y|X,\Z}(y,x,\z;\bb)f_{C,X|\Z}(c,x,\z)dx}
f_{Y|X,\Z}(y,x,\z;\bb)I(c<x) dy\n\\
&=&I(x\le c)\frac{E\{\a_0(C,x,\z)I(x\le C)\mid x,\z\}}{E\{I(x\le C)\mid x,\z\}}\n\\
&&+I(c<x)E\left[\frac{E\{\a_0(c,X,\z)I(X> c)\mid Y, c,\z\}}{E\{I(X>
  c)\mid Y, c,\z\}}\mid c, x,\z\right].
\ese

Since $E\{\S_\bb^F(Y,w,\Delta,\z) \mid c,x,\z\}=0$, we have
\bse
&&I(c<x)E\left[\frac{E\{\S_\bb^F(Y,X,\z)I(X> c)\mid Y, c,\z\}}{E\{I(X>
	c)\mid Y, c,\z\}}\mid c, x,\z\right]\n\\
&=&I(x\le c)\frac{E\{\a_0(C,x,\z)I(x\le C)\mid x,\z\}}{E\{I(x\le C)\mid x,\z\}}\\
&&+I(c<x)E\left[\frac{E\{\a_0(c,X,\z)I(X> c)\mid Y, c,\z\}}{E\{I(X>
	c)\mid Y, c,\z\}}\mid c, x,\z\right].
\ese

To further simplify, we get:
\be\label{eq:a}
0&=&I(x\le c)\frac{E\{\a_0(C,x,\z)I(x\le C)\mid x,\z\}}{E\{I(x\le C)\mid x,\z\}}+\n\\
&&I(c<x)E\left[\frac{E\{(\a_0(c,X,\z)-\S_\bb^F(Y,X,\z))I(X> c)\mid Y, c,\z\}}{E\{I(X>
	c)\mid Y, c,\z\}}\mid c, x,\z\right].\n\\
\ee

Thus,
\bse
\S\eff(y,w,\delta,\z)&=&
\S_\bb(y,w,\delta,\z)
-\delta
\frac{\int_x^\infty\a_0(c,x,\z)f_{C,X|\Z}(c,x,\z)dc}{\int_x^\infty f_{C,X|\Z}(c,x,\z)dc}\\
&&-(1-\delta)\frac{
\int_{c}^{\infty}\a_0(c,x,\z)f_{Y|X,\Z}(y,x,\z;\bb)f_{C,X|\Z}(c,x,\z)dx
}{\int_{c}^{\infty}f_{Y|X,\Z}(y,x,\z;\bb)f_{C,X|\Z}(c,x,\z)dx}\\
&=&
\S_\bb(y,w,\delta,\z)
-\delta 
\frac{E\{\a_0(C,x,\z)I(x\le C)\mid x,\z\}}{E\{I(x\le C)\mid x,\z\}}\\
&&-(1-\delta) \frac{E\{\a_0(c,X,\z)I(X> c)\mid Y, c,\z\}}{E\{I(X>
  c)\mid Y, c,\z\}},
\ese
where $\a_0(c,x,\z)$ satisfies (\ref{eq:a}).

\subsection{Proof of Theorem \ref{thm:consis}}\label{sec:proofthmconsis}

 By conditions (C2)--(C4), we have
\be\label{eq:uniform_conver}
\sup_{\bb \in \mathcal{B}} \Vert n^{-1}\sum_{i=1}^{n}\S\eff^*(y_i,w_i,\delta_i,\z_i;\bb)-E\{\S\eff^*(Y,W,\Delta,\Z;\bb)\} \Vert \rightarrow 0
\ee
in probability. Let $\wh
  Q_n(\bb)\equiv-\|n^{-1}\sum_{i=1}^{n}\S\eff^*(y_i,w_i,\delta_i,\z_i;\bb)\|^2$
  and $Q_0(\bb)\equiv-\|E\{\S\eff^*(Y,W,\Delta,\Z;\bb)\|^2$.
From conditions (C1)-(C4), it follows that $Q_0(\bb_0)=0$ and $Q_0(\bb)$ is uniquely maximized at $\bb_0$, while $Q_n(\wh \bb_n)=0$ holds by the definition of $\wh\bb_n$. Then for any $\epsilon > 0$, by  (\ref{eq:uniform_conver}) and the continuous mapping theorem, we have, with probability approaching one,
\bse
0 \geq Q_0(\wh \bb_n) > -\epsilon.
\ese
Let $\mathcal{N}$ be any open set of $\mathcal{B}$ containing
$\bb_0$. By the compactness of $\mathcal{B} \cap
  \mathcal{N}^c$, condition (C3), because $Q_0(\bb)$ is uniquely maximized at $\bb_0$, we have $\sup_{\bb \in \mathcal{B} \cap \mathcal{N}^c} Q_0(\bb)=Q_0(\bb^*) < Q_0(\bb_0)=0$ for some $\bb^* \in \mathcal{B} \cap \mathcal{N}^c$. 

Thus, choosing $\epsilon = -Q_0(\bb^*)$, it follows that $Q_0(\wh \bb_n) > \sup_{\bb \in \mathcal{B} \cap \mathcal{N}^c} Q_0(\bb)$ with probability approaching one. Hence, $\wh \bb_n \in \mathcal{N}$, i.e., $\wh \bb_n \rightarrow \bb_0$ in probability.

\subsection{Proof of Theorem \ref{thm:norm}}\label{sec:proofthmnormal}

By Taylor's theorem,
\bse
0=n^{-1}\sum_{i=1}^n\S\eff^*(y_i,w_i,\delta_i,\z_i;\wh{\bb}_n)=n^{-1}\sum_{i=1}^n\S\eff^*(y_i,w_i,\delta_i,\z_i;\bb_0)+J_n^*(\boldsymbol \xi)(\wh{\bb}_n-\bb_0)
\ese
for some $\boldsymbol \xi$ on the line joining $\bb_0$ and
$\wh{\bb}_n$. By Theorem \ref{thm:consis}, we have $\wh{\bb}_n
\rightarrow \bb_0$ in probability,
thus $\boldsymbol \xi \rightarrow \bb_0$ in probability. Combining $\boldsymbol \xi \rightarrow \bb_0$ in probability with condition (C7), we have
\bse
J_n^*(\boldsymbol \xi)=J^*(\bb_0)+o_P(1),
\ese
where $o_P(1)$ means a matrix sequence whose Frobenius norm tends to 0. Hence
\bse
0=n^{-1}\sum_{i=1}^n\S\eff^*(y_i,w_i,\delta_i,\z_i;\bb_0)+J^*(\bb_0)(\wh{\bb}_n-\bb_0)+o_P(1)(\wh{\bb}_n-\bb_0).
\ese
By condition (C8), we have
\bse
(\wh{\bb}_n-\bb_0)=-J^*(\bb_0)^{-1}n^{-1}\sum_{i=1}^n\S\eff^*(y_i,w_i,\delta_i,\z_i;\bb_0)+J^*(\bb_0)^{-1}o_P(1)(\wh{\bb}_n-\bb_0).
\ese
Rearranging this equation, we get
\bse
\sqrt{n}\{I_q+o_P(1)\}(\wh{\bb}_n-\bb_0)=-J^*(\bb_0)^{-1}n^{-1/2}\sum_{i=1}^n\S\eff^*(y_i,w_i,\delta_i,\z_i;\bb_0).
\ese
By the central limit theorem,
\bse
n^{-1/2}\sum_{i=1}^n\S\eff^*(y_i,w_i,\delta_i,\z_i;\bb_0) \rightarrow N\{\0,V^*(\bb_0)\}
\ese
in distribution, where $V^*(\bb_0)=E\{\S\eff^*(Y,W,\Delta,\Z;\bb_0)^{\otimes2}\}$. Hence, by Slutsky's Theorem, 
\bse
\sqrt{n}\{I_q+o_P(1)\}(\wh{\bb}_n-\bb_0) \rightarrow N[\0,\{J^*(\bb_0)\}^{-1}V^*(\bb_0)\{J^*(\bb_0)\}^{-\trans}]
\ese
in distribution. Using Slutsky's Theorem again, we have
\bse
\sqrt{n}(\wh{\bb}_n-\bb_0) \rightarrow N[\0,\{J^*(\bb_0)\}^{-1}V^*(\bb_0)\{J^*(\bb_0)\}^{-\trans}]
\ese
in distribution.
\qed

\subsection{Proof of Theorem \ref{thm:chisq}}\label{proofofthmchisq}
Under regularity conditions (C1)--(C8) and the null hypothesis, we have
\bse
\sqrt{n}(\wh\bb_1-\bb)=\frac{1}{\sqrt{n}}\sum\limits_{i=1}^n \bphi_1(y_i,w_i,\delta_i,\z_i;\bb)+\bo_P(1), \\
\sqrt{n}(\wh \bb_2-\bb)=\frac{1}{\sqrt{n}}\sum\limits_{i=1}^n \bphi_2(y_i,w_i,\delta_i,\z_i;\bb)+\bo_P(1).
\ese
Here, $\bphi_1$ and $\bphi_2$ are the influence functions of $\wh\bb_1$ and $\wh\bb_2$, respectively. Specifically, $\bphi_i=-[E\{\partial\S_i(Y,W,\Delta,\Z;\bb)/\partial
          \bb\trans\}]^{-1}\S_i(Y,W,\Delta,\Z;\bb)$, for $i=1,2$,
          where $\S_1$ is $\S_{\rm{CC}}$, $\S_{\rm{Inv}}^*$, or $\S\eff$, and $\S_2$ is
          $\S_{\rm{Imp}}$. 
Similarly, under regularity conditions (C1)--(C8) and the alternative hypothesis:
\bse
\sqrt{n}(\wh \bb_1-\bb)&=&\frac{1}{\sqrt{n}}\sum\limits_{i=1}^n \bphi_1(y_i,w_i,\delta_i,\z_i;\bb)+\bo_P(1), \\
\sqrt{n}(\wh \bb_2-\bb-\bxi)&=&\frac{1}{\sqrt{n}}\sum\limits_{i=1}^n \bphi_2(y_i,w_i,\delta_i,\z_i;\bb)+\bo_P(1).
\ese
Here, $\bxi$ ($\neq \0$) represents the non-zero bias introduced by the imputation estimator, while $\bphi_1$ and $\bphi_2$ are defined as before.

Thus, 
\bse
\sqrt{n}(\wh\bb_1-\wh\bb_2-\bxi)=\frac{1}{\sqrt{n}}\sum\limits_{i=1}^n\{\bphi_1(y_i,w_i,\delta_i,\z_i;\bb)-\bphi_2(y_i,w_i,\delta_i,\z_i;\bb)\}+\bo_P(1).
\ese
Here, $\bxi=\0$ under the null hypothesis and $\bxi\neq \0$ under the alternative hypothesis. Consequently, we have
\bse
n(\wh\bb_1-\wh\bb_2)^{\trans} V^{-1} (\wh\bb_1-\wh\bb_2) \rightarrow \chi^2_p(\|\bxi\|^2).
\ese
Here, $\chi^2_p(\|\bxi\|^2)$ is a noncentral chi-square distribution with $p$ degrees of freedom and noncentrality parameter $\|\bxi\|^2$ (the square of $l_2$-norm of $\bxi$). 
\qed

\subsection{Additional Simulation Results}
\label{sec:addtional-sims}

\begin{table}[H]
	\renewcommand{\arraystretch}{1.5}
	\begin{tabular}{ccllllllll}
		\hline
		&               & \multicolumn{4}{c}{10\% censoring}         & \multicolumn{4}{c}{50\% censoring} \\ \cline{3-10} 
		Working  &      & & & &  &    &     &     &   \\ 
		model & Estimator     & \multicolumn{1}{c}{Mean} & ESE & ASE & Cov & Mean    & ESE    & ASE    & Cov    \\ \hline

        tru          & SPIRE      &          0.4988                &    0.0668 &  0.0651  &   94.3\%  &  0.5005       &   0.0749    &    0.0732    &    94.5\%    \\
		& CC &           0.5007               & 0.0676    &  0.0659   &   94.4\%  &   0.5015      &    0.0754    &   0.0736     &    94.5\%    \\
		& IPW &             0.4997             &  0.0770   &   0.0740  &    94.6\% &    0.5004     &    0.0937    &    0.0911    &    93.3\%    \\
		& MLE    &         0.4984                 &   0.0575  &  0.0567   &  95.3\%   &   0.4920     &    0.0444   &    0.0448    &     94.3\%   \\
		mis       & SPIRE      &          0.5007                &   0.0676  &   0.0659  &    94.3\% &    0.5015     &    0.0754    &     0.0737   &     94.5\%   \\
		& CC &        0.5007                  &   0.0676  &  0.0659   &  94.4\%   &    0.5015     &    0.0754    &    0.0736    &   94.5\%     \\
		& IPW &            0.4985              &   0.0988  &   0.0970  &  96.4\%   &     0.5009    &    0.1059    &  0.1015      &   93.4\%     \\
		& MLE    &          0.4671                &  0.0536   &    0.0530 &   90.1\%  &     0.4308    &  0.0478      &    0.0486   &    70.8\%    \\ \hline
		&               & \multicolumn{4}{c}{70\% censoring}         & \multicolumn{4}{c}{80\% censoring} \\ \cline{3-10} 
		Working  &      & & & &  &    &     &     &   \\ 
        model & Estimator     & Mean                     & ESE & ASE & Cov & Mean    & ESE    & ASE    & Cov    \\ \hline
		tru          & SPIRE      &       0.5003           &  0.0943   &  0.0942   &   95.8\%  &    0.5020     &   0.1201     &   0.1205     &   95.0\%     \\
		& CC &          0.5009                &    0.0946 &   0.0951  &  96.0\%   &     0.5046    &    0.1215    &  0.1204      &   94.7\%     \\
		& IPW &         0.4935               &   0.1179  &  0.1188   &  95.3\%   &    0.5010     &    0.1606    &    0.1542    &    94.1\%    \\
		& MLE    &          0.4881                &   0.0480  & 0.0490   &   94.8\%  &     0.4849    &     0.0554   &   0.0567     &   94.1\%     \\
		mis       & SPIRE      &           0.5009               &   0.0945  & 0.0947    &  95.9\%   &    0.5046    &     0.1215   &   0.1197     &    94.7\%    \\
		& CC &         0.5009                 &  0.0946   &  0.0951   & 96.0\%    & 0.5046       &    0.1215    &     0.1204   &    94.7\%    \\
		& IPW &       0.4914                  &  0.1229   &  0.1228  &  95.5\%   &    0.5019     &   0.1606     &     0.1596   &   93.9\%     \\
		& MLE    &        0.4225                  &   0.0586  &   0.0604  &   75.9\%  &    0.4151     &     0.0725   &   0.0736     &    79.9\%    \\ \hline
	\end{tabular}
	\caption{Simulation results of $\beta_0$ in the controlled setting based on $N=1,000$          replicates.  All abbreviations and definitions are as in Table \ref{table:b1s1}.  
    }
		\label{table:b0s1}
\end{table}

\begin{table}[H]
	\renewcommand{\arraystretch}{1.5}
	\begin{tabular}{ccllllllll}
		\hline
		&               & \multicolumn{4}{c}{10\% censoring}         & \multicolumn{4}{c}{50\% censoring} \\ \cline{3-10} 
		Working  &      & & & &  &    &     &     &   \\ 
		 model & Estimator     & \multicolumn{1}{c}{Mean} & ESE & ASE & Cov & Mean    & ESE    & ASE    & Cov    \\ \hline
		tru          & SPIRE      &          -0.1993                &    0.0866 &  0.0833  &   93.8\%  &  -0.2008      &   0.1216    &    0.1187    &    93.9\%    \\
		& CC &           -0.2006               & 0.0871    &  0.0838   &   93.9\%  &   -0.2022      &    0.1224   &   0.1195     &    93.9\%    \\
		& IPW &             -0.1997           &  0.0913   &   0.0888  &    93.9\% &    -0.2002     &    0.1548    &    0.1515    &    93.6\%    \\
		& MLE    &         -0.1988                 &   0.0810  &  0.0786   &  94.3\%   &   -0.1938     &    0.0829   &    0.0803    &     94.4\%   \\
		mis       & SPIRE      &          -0.2006                &   0.0871  &   0.0838  &    93.8\% &    -0.2022     &    0.1224   &     0.1197   &     93.9\%   \\
		& CC &        -0.2006                  &   0.0871  &  0.0838   &  93.9\%   &    -0.2022    &    0.1224   &    0.1195    &   93.9\%     \\
		& IPW &            -0.1983              &   0.1160  &   0.1143  &  95.5\%   &     -0.2011    &    0.1767    &  0.1704      &   93.0\%     \\
		& MLE    &          -0.1669                &  0.0770   &    0.0751 &   91.7\%  &     -0.0792    &  0.0682      &    0.0670   &    56.3\%    \\ \hline
		&               & \multicolumn{4}{c}{70\% censoring}         & \multicolumn{4}{c}{80\% censoring} \\ \cline{3-10} 
		Working  &      & & & &  &    &     &     &   \\ 
        model & Estimator     & Mean                     & ESE & ASE & Cov & Mean    & ESE    & ASE    & Cov    \\ \hline
		tru          & SPIRE      &       -0.2018           &  0.1616   &  0.1585  &   94.4\%  &    -0.2042     &   0.1990     &   0.2033     &   95.4\%     \\
		& CC &          -0.2028               &    0.1631 &   0.1609  &  94.6\%   &     -0.2107    &    0.2071    &  0.2031      &   94.5\%     \\
		& IPW &         -0.1886               &   0.2171  &  0.2160   &  94.8\%   &    -0.2007     &    0.3032    &    0.2903    &    93.7\%    \\
		& MLE    &          -0.1951               &   0.0856  & 0.0835   &   93.9\%  &     -0.1991    &     0.0903   &   0.0877     &   93.7\%     \\
		mis       & SPIRE      &           -0.2028              &   0.1631  & 0.1606    &  94.6\%   &    -0.2107   &     0.2071   &   0.2012     &    94.5\%    \\
		& CC &         -0.2028                &  0.1631   &  0.1609   & 94.6\%    & -0.2107       &    0.2071    &     0.2031   &    94.5\%    \\
		& IPW &       -0.1860                 &  0.2216   &  0.2176  &  94.6\%   &    -0.2047     &   0.3032     &     0.2813   &   94.2\%     \\
		& MLE    &        -0.0576                &   0.0668  &   0.0657  &   42.8\%  &   -0.0509     &     0.0665   &   0.0656    &    38.8\%    \\ \hline
	\end{tabular}
	\caption{
    Simulation results of $\beta_2$ in the controlled setting based on $N=1,000$          replicates.  All abbreviations and definitions are as in Table \ref{table:b1s1}.
    }
		\label{table:b2s1}
\end{table}

\begin{table}[H]
\scalebox{0.9}{ 
	\renewcommand{\arraystretch}{1.5}
	\begin{tabular}{ccllllllll}
		\hline
		&               & \multicolumn{4}{c}{$\beta_0$}         & \multicolumn{4}{c}{$\beta_2$} \\ \cline{3-10} 
		Working  &      & & & &  &    &     &     &   \\ 
 model & Estimator     & \multicolumn{1}{c}{Mean} & ESE & ASE & Cov & Mean    & ESE    & ASE    & Cov    \\ \hline
		tru          & SPIRE      &          1.3020         &    0.1898 &  0.1899  &   94.4\%  &  -1.5104      &   0.4096    &    0.4067    &    95.8\%    \\
		& CC &           1.3034        & 0.1924    &  0.1896   &   94.4\%  &   -1.5138      &    0.4097  &   0.4049     &    95.6\%    \\
		& IPW &            1.3014          &  0.2116   &   0.2030  &    93.7\% &    -1.5208     &    0.4562    &    0.4469   &    95.3\%    \\
		& MLE    &         1.3017       &   0.1049  &  0.1036   &  95.0\%   &   -1.5051     &    0.1842   &    0.1787    &     93.9\%   \\
		unif       & SPIRE      &          1.3079          &   0.1908  &   0.1902  &    94.7\% &    -1.5125     &    0.4108  &     0.4054   &     95.7\%   \\
		& CC &      1.3034        & 0.1924    &  0.1896   &   94.4\%  &   -1.5138      &    0.4097  &   0.4049     &    95.6\%      \\
		& IPW &            1.2963      &   0.2424    &   0.2458    &  94.3\%   &     -1.4709    &    0.5355   &  0.5246      &   94.4\%     \\
		& MLE    &          1.3479                &  0.1022   &    0.1018 &   92.6\%  &     -1.5267   &  0.1847      &    0.1793   &    94.3\%    \\ 
		K-M       & SPIRE      &          1.3076          &   0.1908  &   0.1888  &    94.5\% &    -1.5125     &    0.4100  &     0.4026   &     95.6\%   \\
		& CC &      1.3034        & 0.1924    &  0.1896   &   94.4\%  &   -1.5138      &    0.4097  &   0.4049     &    95.6\%      \\
		& IPW &            1.2963      &   0.2424    &   0.2458    &  94.3\%   &     -1.4709    &    0.5355   &  0.5246      &   94.4\%     \\
		& MLE    &          1.4906       &  0.1461   &    0.1521 &   76.6\%  &     -1.7237  &  0.1999      &    0.2122   &    95.4\%    \\ \hline
		&               & \multicolumn{4}{c}{$\beta_3$}         &  \multicolumn{4}{c}{$\sigma^2$}   \\ \cline{3-10} 
		Working  &      & & & &  &    &     &     &   \\ 
model & Estimator     & Mean                     & ESE & ASE & Cov & Mean    & ESE    & ASE    & Cov  \\ \hline
		tru          & SPIRE     &       0.0983           &  0.1991   &  0.2134 &   94.7\%  & 0.9894 & 0.0654 & 0.0685 & 96.6\%  \\
		& CC &         0.0984     &    0.2003 &   0.2135  &  94.6\%  & 0.9899 & 0.0658 & 0.0621 & 94.3\%     \\
		& IPW &         0.1183               &   0.4016  &  0.4156   &  97.6\%  & 0.9859 & 0.0712 & 0.0658 & 92.6\%  \\
		& MLE    &          0.0930               &   0.1542 & 0.1573  &   95.1\%    & 0.9979 & 0.0296 & 0.0297 & 95.4\%   \\
		unif       & SPIRE      &           0.0938  &   0.1987  & 0.2140    &  94.3\%   &  0.9909 & 0.0655 & 0.0622 & 95.0\%   \\
		& CC &    0.0984     &    0.2003 &   0.2135  &  94.6\%  & 0.9899 & 0.0658 & 0.0621 & 94.3\%    \\
		& IPW &       0.0918      &  0.2622  &  0.2641  &  94.4\%  & 0.9589 & 0.1015 & 0.0953 & 91.5\%     \\
		& MLE   &      0.1190   &   0.1233  &   0.1236  &   95.1\%  & 0.9688 & 0.0320 & 0.0318 & 83.2\%   \\ 
		K-M       & SPIRE      &    0.0941  &   0.1986  & 0.2132    &  94.3\%   &  0.9910 & 0.0653 & 0.0620 & 94.4\%   \\
		& CC &    0.0984     &    0.2003 &   0.2135  &  94.6\%  & 0.9899 & 0.0658 & 0.0621 & 94.3\%    \\
		& IPW &       0.0918      &  0.2622  &  0.2641  &  94.4\%  & 0.9589 & 0.1015 & 0.0953 & 91.5\%     \\
		& MLE   &      0.0352   &   0.1730  &   0.1750 &   93.6\%  & 1.0357 & 0.0306 & 0.0316 & 79.9\%   \\ \hline
	\end{tabular}
    }
	\caption{
    %
      Simulation results of $\beta_0$, $\beta_2$, $\beta_3$, and $\sigma^2$ in the realistic setting based on $N=1,000$          replicates.  All abbreviations and definitions are as in Table \ref{table:b1s2}.
    }
    	\label{table:b2s2}
\end{table}

\end{document}